\documentclass{article}
\usepackage{pdfsync}
\usepackage{amssymb,amsmath,latexsym,amsfonts,amsthm}
\usepackage{epstopdf}
\usepackage[mathscr]{eucal}
\DeclareMathOperator{\rank} {rank}
\DeclareMathOperator{\de}{d}

\DeclareMathOperator{\orb}{orb}
\DeclareMathOperator{\dom}{dom}
\DeclareMathOperator{\Sp}{Sp}

\DeclareMathOperator{\dist}{dist}

\begin{document}

\title{Vertex partitions of metric spaces with finite distance sets}
\author{ N. W. Sauer}

\date{January 04 2010 }
\maketitle

\newcommand{\snl} {\\ \smallskip}
\newcommand{\mnl}{\\ \medskip} 
\newcommand{\Bnl} {\\ Bigskip}
\newcommand{\edge}{\makebox[22pt]{$\circ\mspace{-6 mu}-\mspace{-6 mu}\circ$}}
\newcommand{\nedge}{\makebox[22pt]{$\circ\mspace{-6 mu}\quad\: \mspace{-6 mu}\circ$}}

\newcommand{\restrict}[2]{#1\mspace{-2mu}\mathbin{\upharpoonright}\mspace{-1mu} #2}
\newcommand{\Kat}{Kat\v{e}tov }
\newcommand{\Fra}{Fra\"{\i}ss\'e}
\newcommand{\str}[1]{\stackrel{#1}{\sim}}

\newtheorem{thm}{Theorem}[section]
\newtheorem{lem}{Lemma}[section]
\newtheorem{coroll}{Corollary}[section]
\newtheorem{ass}{Assumption}[section]
\newtheorem{defin}{Definition}[section]
\newtheorem{example}{Example}[section] 
\newtheorem{fact}{Fact}[section]

\newtheorem{prop}{Proposition}[section]
\newtheorem{obs}{Observation}[section]
\newtheorem{cor}{Corollary}[section]
\newtheorem{sublem}{Sublemma}[section]
\newtheorem{claim}{Claim}[section]
\newtheorem{question}{Question}[section] 

\newtheorem{comment}{Comment}[section]
\newtheorem{problem}{Problem}[section]
\newtheorem{remark}{Remark}[section]

\newcommand\con{\char'136{}}

\begin{abstract}
A metric space $\mathrm{M}=(M,\de)$ is {\em indivisible} if for every colouring $\chi: M\to 2$ there exists $i\in 2$ and a copy $\mathrm{N}=(N, \de)$ of $\mathrm{M}$ in $\mathrm{M}$ so that $\chi(x)=i$ for all $x\in N$. The metric space $\mathrm{M}$ is {\em homogeneus} if for every isometry $\alpha$ of a finite subspace of $\mathrm{M}$ to a subspace of $\mathrm{M}$ there exists an isometry of $\mathrm{M}$ onto $\mathrm{M}$ extending $\alpha$. A homogeneous metric space $\mathrm{U}_{\mathcal{D}}$  with $\mathcal{D}$ as set of distances  is an {\em Urysohn metric space} if every finite metric space with set of distances a subset of $\mathcal{D}$ has an isometry into $\mathrm{U}_{\mathcal{D}}$.   The main result of this paper states that all countable  Urysohn metric spaces with a finite set of distances are indivisible. 
\end{abstract}

\vfil\eject
\section{Introduction}

The connection between structural Ramsey theory,  {\Fra} Theory  and topological dynamics established in \cite{Pe1} and \cite{KPT} leads naturally to the partition problem addressed in this paper. See \cite{NVT} for a more extensive discussion and for establishing claimed facts. Detailed introductions to {\Fra} limits can be found in \cite{Fr} or \cite{WH}. Below is a short account of the  origin of the problem discussed in this paper.

For $(M;\de)$ a metric space let $\dist(M;\de)$ be the set of  distances between points of $(M;\de)$.  A metric space $\mathrm{M}$  is {\em homogeneous} if for every isometry $\alpha$ of a finite subspace of $\mathrm{M}$ to a subspace of $\mathrm{M}$ there exists an isometry of $\mathrm{M}$ onto $\mathrm{M}$ extending $\alpha$. A homogeneous metric space $\boldsymbol{U}_{\mathcal{D}}$ with $\dist(\boldsymbol{U}_{\mathcal{D}})=\mathcal{D}$     is an {\em Urysohn metric space}  if every finite metric space $\mathrm{M}$ with $\dist(\mathrm{M})\subseteq \mathcal{D}$   has an isometry into $\boldsymbol{U}_{\mathcal{D}}$. As Urysohn metric spaces are universal objects, a  subset $\mathcal{D}\subseteq \Re_{\geq 0}$ is called {\em universal} if there exists an Urysohn metric space $\boldsymbol{U}_{\mathcal{D}}$. To decide whether a given set $\mathcal{D}\subseteq \Re_{\geq 0}$ is universal can be difficult. A particular example of an Urysohn space is the {\em Urysohn sphere} $\boldsymbol{U}_{\Re\cap[0,1]}$.   Other examples are the Urysohn spaces $\boldsymbol{U}_m$ with $m\in \omega$ for which $\dist(\boldsymbol{U}_m)$ is equal to $\{0,1,2,\dots,m-1\}$ and the Hilbert space $\ell_2$.

A {\em copy} of a metric space $(M;\de)$ in $(M;\de)$ is the image of an isometry of $(M;\de)$ in $(M;\de)$.  The space $\mathrm{M}=(M;\de)$ is {\em indivisible} if for  every colouring $\chi: M\to 2$ there exists $i\in 2$ and a copy $\mathrm{M}^\ast$ of $\mathrm{M}$ in $\mathrm{M}$ so that $\chi(x)=i$ for all $x\in M^\ast$.

A metric space $(M;\de)$ is {\em oscillation stable} if for every bounded and uniformly continuous function $f: M\to \Re$ and every $\epsilon>0$  there is a copy $(M^\ast,\de)$ of $(M;\de)$ in $(M;\de)$ so that:
\[
\sup\{|f(x)-f(y)| \mid x,y\in M^\ast\}<\epsilon.
\]
The question whether the unit sphere of the Hilbert space  $\ell_2$ is oscillation stable had been known as the {\em distortion problem} and was finally resolved in the negative, see \cite{OS}. This then led to the question whether the other prominent bounded metric space with a large isometry group,  namely the Urysohn sphere  $\boldsymbol{U}_{\Re\cap[0,1]}$,  is oscillation stable. After an initial reformulation of the problem by V. Pestov, see \cite{Pe1},  Lopez-Abad and Nguyen Van Th{\'e}, see \cite{LANVT},  started a programme to reduce the problem to one of discrete  mathematics. In particular they proved that the Urysohn sphere will be oscillation stable if and only if each  of the Urysohn spaces $\boldsymbol{U}_m$ is indivisible. Subsequently it was established in \cite{LNSA} that all of the Urysohn spaces $\boldsymbol{U}_m$ are indivisible, finishing the proof that the Urysohn sphere $\boldsymbol{U}_{\Re\cap[0,1]}$ is oscillation stable.

Which metric spaces are oscillation stable? There does not seem to be any way at present to attack this question in general. Even to ask for a characterization of  the oscillation stable homogeneous metric spaces is beyond our present means. But, to find a characterization of the oscillation stable Urysohn metric spaces might just be possible following the ideas of Lopez-Abad and Nguyen Van Th{\'e} to reformulate the problem as a problem of discrete mathematics. There are essentially two steps to such a characterization. Step 1 is to investigate whether the Urysohn metric spaces $\boldsymbol{U}_{\mathcal{D}}$ for $\mathcal{D}$ finite are indivisible. An, as I think, attractive question in its own right, belonging to the general area of structural Ramsey theory. The present paper contains the proof that all Urysohn metric spaces $\boldsymbol{U}_{\mathcal{D}}$ with $\mathcal{D}$ finite are indivisible, see Theorem \ref{thm:final}.

Step 2 is  to establish a connection between the oscillation stability of  Urysohn spaces $\boldsymbol{U}_{\mathcal{D}}$ for general bounded $\mathcal{D}$ and the indivisibility of the ones for finite $\mathcal{D}$, following \cite{LANVT}.  This connection between the two problems  is to appear in a forthcoming paper. If $\mathcal{D}$ is not bounded $\boldsymbol{U}_{\mathcal{D}}$ can not be oscillation stable. (Unpublished, but generally  known in the area,      the proof being an easy modification of the proof  of Theorem~3.14 in \cite{DLPS}.)

The objects discussed in this paper are metric spaces with a finite set  of distances. They may be viewed as labelled graphs or relational structures. Relational structures are usually denoted by the same letter as their base sets, just using a different font.  On the other hand, metric spaces carry a natural topology and topological spaces are often just denoted by their base sets. Here $\mathrm{M}=(M;\de_\mathrm{M})$ is the full description of a metric space, used mostly  if different metrics on $M$ are needed. This description will often be abbreviated to $(M;\de)$ or just $M$ if the metric is given by context.

\section{Notation and basic facts}

\noindent
Let $0\in \mathcal{D}\subseteq \Re_{\geq 0}$ be a given finite set of numbers.

\vskip 3pt
A pair $\mathrm{H}=(H,\de)$ is a {\em $\mathcal{D}$-graph} if $\de: H^2\to \mathcal{D}$  is a function with $\de(x,y)=0$ if and only if $x=y$ and $\de(x,y)=\de(y,x)$ for all $x,y\in H$.  For $A\subseteq H$ we denote by  $\restrict{\mathrm{H}}{A}$ the {\em substructure of\/ $\mathrm{H}$ generated by  $A$}, that is  the $\mathcal{D}$-graph on $A$ with distance function the restriction of\/ $\de$ to $A^2$. The $\mathcal{D}$-graph $\mathrm{H}$ is {\em metric} if it is a metric space. That is if $\de(p,q)\leq \de(p,r)+\de(r,q)$  for all  triples $(p,q,r)\in H^3$. Let $\mathfrak{M}_{\mathcal{D}}$ be the class of metric spaces $\mathrm{M}$ with $\dist(\mathrm{M})\subseteq \mathcal{D}$ and $\mathfrak{U}_{\mathcal{D}}$ the class of countable Urysohn metric spaces in $\mathfrak{M}_{\mathcal{D}}$. The following 4-values condition, see \cite{NVT} or \cite{DLPS} provides a characterization of universal sets of numbers. 

\begin{lem}\label{lem:triangleen} 
The set $\mathcal{D}$ of numbers is {\em universal}  if and only if for any two triangles  $a_0,b,c$ and $a_1,b,c$  which are in $\mathfrak{M}_{\mathcal{D}}$,   there exists $0<t\in \mathcal{D}$ so that the $\mathcal{D}$-graph  $\mathrm{L}=(\{a_0,a_1,b,c\}, \de)$ with $\de(a,a_1)=t$ and for which the two triangles are induced subspaces,  is an element of $\mathfrak{M}_{\mathcal{D}}$.  (The case $\de(b,c)=0$, that is $b=c$ is included.)
\end{lem}

Let $\mathcal{D}$ be a universal set of numbers.  By scaling $\mathcal{D}$ to $t\mathcal{D}$ for some positive real $t$ the set of distances $t\mathcal{D}$ is a universal set of numbers.  The metric spaces in  $\mathfrak{U}_{\mathcal{D}}$ are indivisible if and only if the metric spaces in $\mathfrak{U}_{t\mathcal{D}}$ are indivisible. Hence we may assume that $\min(\mathcal{D}\setminus\{0\})=1$.

A  function $\mathfrak{t} : F\to \mathcal{D}$, with $F$ a finite subset of $H$, is a  {\em type function} of $\mathrm{H} $. For $\mathfrak{t} $ a type function  let $\Sp(\mathfrak{t})$ be the $\mathcal{D}$-graph on $F\cup \{\mathfrak{t}\}$ for which: 
\begin{enumerate}
\item $\restrict{\Sp(\mathfrak{t})}{\dom(\mathfrak{t})}=\restrict{\mathrm{H} }{\dom(\mathfrak{t})}$.
\item $\forall x\in F\, \,\big(\de(\mathfrak{t},x)=\de(x,\mathfrak{t})=\mathfrak{t}(x)\big)$.
\end{enumerate}
For $\mathfrak{t}$ a type function of $\mathrm{H}$    let
\[
\orb(\mathfrak{t})=\{y\in H\setminus\dom(\mathfrak{t}) : \forall\, x\in  \dom(\mathfrak{t})\, \bigl(\de(y,x)=\de(\mathfrak{t},x)=\mathfrak{t}(x)\bigr)\},
\]
the {\em orbit} of $\mathfrak{t}$. If the distinction is necessary we will write $\orb_{\mathrm{H}}(\mathfrak{t})$. Note that if $\dom(\mathfrak{t})=\emptyset$ then $\orb(\mathfrak{t})=H$.   If $\orb\mathfrak{t}\not=\emptyset$ then $\mathfrak{t}$ {\em can be realized in $\mathrm{H}$} and every $p\in \orb\mathfrak{t}$ is a {\em realization of $\mathfrak{t}$}. A type function $\mathfrak{t}$ of $\mathrm{H}$ is a {\em \Kat function} if $\Sp(\mathfrak{t})$ is metric. Note here that $\mathcal{D}$ is given. Hence $\dist\big(\Sp(\mathfrak{t})\big)\subseteq \mathcal{D}$ is assumed.  The {\em rank} of\/ $\mathfrak{t}$, $\rank(\mathfrak{t})$, is $\min\{\mathfrak{t}(x) : x\in \dom(\mathfrak{t})\}$.

Let $\mathrm{M}\in \mathfrak{M}_{\mathcal{D}}$.  Then a type function  $\mathfrak{t}$ of $\mathrm{M}$ is a \Kat function if and only if  for all $x,y\in \dom(\mathfrak{t})$:
\begin{align}
|\mathfrak{t}(x)-\mathfrak{t}(y)|\leq d(x,y)\leq \mathfrak{t}(x)+\mathfrak{t}(y).
\end{align}
Remember $\mathfrak{t}(x)\in \mathcal{D}$ for all $x\in \dom(\mathfrak{t})$. Because $\dom(\mathfrak{t})\subseteq M$ all triangles in $\dom(\mathfrak{t})$ are metric. Hence,  in order to check that $\Sp(\mathfrak{t})$ is metric it suffices to check all triangles of the form $\{\mathfrak{t},x,y\}$ with $x,y\in \dom(\mathfrak{t})$, which is an easy consequence of Inequalities (1).

The following Theorem \ref{thm:Fra1} is a direct consequence of the general theory of {\Fra} limits see \cite{NVT} for a more extensive discussion or \cite{Fr} and \cite{WH}. The other assertions in the remainder of this section  are known facts as well.  Their proofs  are given. 

\begin{thm}\label{thm:Fra1}
A countable  metric space $\mathrm{M}\in \mathfrak{M}_{\mathcal{D}}$ is an Urysohn space if and only if every \Kat function $\mathfrak{t}$ of \/ $\mathrm{M}$ is realized in $M$. Any two  Urysohn spaces in $\mathfrak{U}_{\mathcal{D}}$ are isometric, indeed:

Let  $\mathrm{M}, \mathrm{N}\in \mathfrak{U}_{\mathcal{D}}$,   then every isometry of a finite subspace of\/  $\mathrm{M}$ to a finite subspace of $\mathrm{N}$ can be extended to an isometry of\/  $\mathrm{M}$ to $\mathrm{N}$. 

Let  $\mathrm{M}\in \mathfrak{U}_{\mathcal{D}}$ and $\mathrm{N}\in \mathfrak{M}_{\mathcal{D}}$ countable. Then every isometry of a finite subspace of\/  $\mathrm{N}$ into  $\mathrm{M}$ can be extended to an isometry of\/  $\mathrm{N}$ into $\mathrm{M}$.
\end{thm}
The last assertion being known as the {\em mapping extension property}. 
For $\mathcal{D}$ universal let  $\boldsymbol{U}_{\mathcal{D}}$ be a particular Urysohn space  in the class  $\mathfrak{U}_{\mathcal{D}}$. 

\begin{cor}\label{cor:Fra1_2}
Let $\mathrm{H}=(H, \de)\in \mathfrak{U}_{\mathcal{D}}$  and $C\subseteq H$.  Then $\restrict{\mathrm{H}}{C}$ is a copy of\/ $\mathrm{H}$ in $\mathrm{H}$ if and only if $\orb(\mathfrak{t})\cap C\not=\emptyset$  for every \Kat function $\mathfrak{t}$ of\/  $\mathrm{H}$ with $\dom(\mathfrak{t})\subseteq C$.
\end{cor}

\begin{thm}\label{thm:orbits}
Let $\mathrm{H}=(H, \de)\in \mathfrak{U}_{\mathcal{D}}$ and  $\mathfrak{t}$ be a \Kat function of\/  $\mathrm{H}$ and  let $\mathcal{D}_\mathfrak{t}=\{n\in \mathcal{D} : n\leq 2\cdot\rank(\mathfrak{t})\}$.  Then the restriction of\/  $\mathrm{H}$ to $\orb(\mathfrak{t})$  is isometric to $\boldsymbol{U}_{\mathcal{D}_\mathfrak{t}}$. (It follows that $\orb(\mathfrak{t})$ is infinite.) 
\end{thm}
\begin{proof}
Let $\mathfrak{f}$ be a type function with $\dom(\mathfrak{f})\subseteq \orb(\mathfrak{t})$ and $\Sp(\mathfrak{f})$ metric and $\dist(\Sp(\mathfrak{f}))\subseteq \mathcal{D}_\mathfrak{t}$. Let $x\in \orb(\mathfrak{t})$ and $\mathfrak{g}$ the type function with $\dom(\mathfrak{g})=\dom(\mathfrak{f})\cup \dom(\mathfrak{t})$ and with $\mathfrak{f}\subseteq \mathfrak{g}$. For every $p\in \orb(\mathfrak{t})$ let $\mathfrak{g}(p)=\mathfrak{t}(x)$. In order to check that $\mathfrak{g}$ is metric, given that $\mathfrak{t}$ and $\mathfrak{f}$ are metric, triangles of the form $\{\mathfrak{g}, a,b\}$  with $a\in \dom(\mathfrak{f})$ and $b\in \dom(\mathfrak{t})$ have to be checked to be metric. Note that $\de(\mathfrak{g}, b)=\de(a,b)$. Hence the triangle is metric because $\de(a,b)\geq \rank(\mathfrak{t})$ and $\de(\mathfrak{g},a)\leq 2\cdot\rank(\mathfrak{t})$. 

It follows that $\mathfrak{g}$ is a \Kat function and hence has a realization $q$ according to Theorem \ref{thm:Fra1}. Then $q\in \orb(t)$ because $\mathfrak{t}\subseteq \mathfrak{g}$. Using Corollary \ref{cor:Fra1_2} with $\mathcal{D}_\mathfrak{t}$ for $\mathcal{D}$, we conclude that $\orb(\mathfrak{t})$ is isomorphic to $\boldsymbol{U}_{\mathcal{D}_\mathfrak{t}}$. 
\end{proof}

\begin{cor}\label{cor:orbits}
Let $\mathfrak{t}$ be a \Kat function of $\boldsymbol{U}_{\mathcal{D}}$ with $rank(\mathfrak{t})=r$ and $s\leq 2r\in \mathcal{D}$ and $x\in \orb(\mathfrak{t})$. Then the type function $\mathfrak{p}$ with $\dom(\mathfrak{p})=\dom(\mathfrak{t})\cup \{x\}$ and $\mathfrak{t}\subseteq \mathfrak{p}$ and $\mathfrak{p}(x)=s$ is a \Kat function.
\end{cor}

\begin{lem}\label{lem:finiterem}
Let $\mathrm{M}=(M, \de)$ be isometric to $ \boldsymbol{U}_{\mathcal{D}}$  and $A$ and $B$ finite subsets of $M$ with $A\cap B=\emptyset$. Then there exists an isomorphism $\alpha$ of\/ $\mathrm{M}$ to $\restrict{\mathrm{M}}{(M\setminus B})$ with $\alpha(a)=a$ for all $a\in A$. 
\end{lem}
\begin{proof}
Let $\restrict{\mathrm{M}}{(M\setminus B})=\mathrm{N}$ and $\mathfrak{t}$ a \Kat  function of $\mathrm{N}$ and hence a \Kat function of $\mathrm{M}$.   Because $\orb_\mathrm{M}(\mathfrak{t})$ is infinite there is a $y\in \orb(\mathfrak{t})\setminus A$.   Hence $\mathfrak{t}$ is realized in $N$, implying $\restrict{\mathrm{M}}{N}$ is isometric to $\mathrm{M}$ according to Corollary \ref{cor:Fra1_2}.  It follows from Theorem \ref{thm:Fra1} that the identity map on $A$ has an extension to an isometry of\/    $\mathrm{M}$ to $\restrict{\mathrm{M}}{N}$. 
\end{proof}

\section{The structure of $\boldsymbol{U}_{\mathcal{D}}$}

\noindent
Let $\mathcal{D}$ be a universal set of numbers  and  $\mathrm{M}=(M;\de)\in \mathfrak{U}_{\mathcal{D}}$.
\vskip 5pt

 For $r$ a positive real   let    $r^{\langle -\rangle}=\max\big([0,r)\cap \mathcal{D}\big)$, the largest number in $\mathcal{D}$ smaller than $r$. For $r<\max\mathcal{D}$ let $r^{\langle +\rangle}=\min\big((r,\max\mathcal{D}]\cap\mathcal{D}\big)$, the smallest number in $\mathcal{D}$ larger than $r$.  For $r=\max\mathcal{D}$ let $r^{\langle +\rangle}=r$. The number $r\in \mathcal{D}$ is a {\em jump number} if $r^{\langle +\rangle}> 2\cdot r$.
 
\begin{lem}\label{lem:reach}
Let $0<m\in \mathcal{D}$ so that the set $S$ of numbers $r\in \mathcal{D}$ with $r^{\langle +\rangle}>m+r$ 
  is not empty. Then $\min S$ is a jump number.
\end{lem} 
\begin{proof}
Let $r=\min S$.   If $m\geq r$ then $r^{\langle +\rangle}>m+r$ implies  $r^{\langle +\rangle}>r+ r$.   Let $m<r$ and assume for a contradiction that $r^{\langle +\rangle}\leq r+r$. Let $s\in \mathcal{D}$ be minimal with $r^{\langle +\rangle}\leq s+r$. Then $s\leq r$ and $0<s$ hence 
\[
0<m\leq s^{\langle -\rangle}<s\leq r<m+r\leq s^{\langle -\rangle}+r<r^{\langle +\rangle}\leq s+r\leq r+r. 
\]
It follows that the largest number in $\mathcal{D}$ less than or equal to $s^{\langle -\rangle}+r$ is $r$.  Because $s^{\langle -\rangle}<r$ it follows from the minimality of $r$ that $s=(s^{\langle -\rangle})^{\langle +\rangle}\leq m+s^{\langle -\rangle}$.  Let  $A=\{a,b,c\}$ and $A'=\{a',b,c\}$ be two triangles  with $\de(a,c)=s^{\langle -\rangle}$ and $\de(b,c)=s$ and $\de(a',c)=r$ and $\de(a,b)=m$ and $\de(a',b)=r^{\langle +\rangle}$. Then $A,A'\in \mathfrak{M}_{\mathcal{D}}$.

If there is a number $t\in \mathcal{D}$ so that the space $a,b,c,a'$ with $\de(a,a')=t$ is an element of $\mathfrak{M}_{\mathcal{D}}$ then $t\leq  r$ because  the triangle $a,a',c$ is metric. On the other hand the triangle $a,a',b$ is metric  and hence $r^{\langle +\rangle}\leq t+m \leq r+m<r^{\langle +\rangle}$. Hence there is no such number $t$ and it follows from Lemma   \ref{lem:triangleen} that $\mathcal{D}$ is not universal.

\end{proof}

\begin{defin}\label{defin:block}
The set $\mathcal{B}\subseteq \mathcal{D}$ is a {\em block} of\/ $\mathcal{D}$ if there exists an enumeration  $(b_i: i\in n+1) $ of\/ $\mathcal{B}$ so that:  
\begin{enumerate}
\item $0<b_i<b_{i+1}$ for all $i\in n$.
\item $b_0>b_0^{\langle -\rangle}+b_0^{\langle -\rangle}$.
\item  $b_{i+1}=b_i^{\langle +\rangle}$ for all $i\in n$.
\item $b_i+b_0\geq b_{i+1}$ for all $i\in n$.
\end{enumerate}
\end{defin}
\noindent
Lemma \ref{lem:reach} implies:
\begin{thm}\label{thm:distset}
The  distance set $\mathcal{D}$  of a universal  metric space   is the union of disjoint blocks $\mathcal{B}_i$ 
 so that: 
 \begin{enumerate}
 \item $x<y$ for all $x\in \mathcal{B}_i$ and $y\in \mathcal{B}_{i+1}$.
 \item $2\cdot \max\mathcal{B}_i<\min\mathcal{B}_{i+1}$.
 \item  $x+\min(\mathcal{B}_i)\geq x^{\langle +\rangle}$ for all $x\in \mathcal{B}_i$.
 \end{enumerate}
 \end{thm}

\vskip 4pt
\noindent
Let $r$ be the maximum of a block $\mathcal{B}$ of $\mathcal{D}$, that is a jump number.
\vskip 3pt

It follows that the relation $\stackrel{r}{\sim}$ on $M$ with $x\stackrel{r}{\sim} y$ iff $\de(x,y)\leq r$ is an equivalence relation and that every automorphism of $\mathrm{M}$ maps elements of $M{/\negthickspace\stackrel{r}{\sim}}$ onto elements of $M{/\negthickspace\stackrel{r}{\sim}}$. Let $A,B,C\in M{/\negthickspace\stackrel{r}{\sim}}$ with $A\not=C\not=B$ and $a\in A$ and $b\in B$ and $c\in C$. Then there exists an automorphism $\alpha$ of $\mathrm{M}$ to $\mathrm{M}$ mapping $a$ to $b$ but $\alpha(c)\not\in B$ because $\de(a,c)> r$ and $\de(b,x)\leq r$ for all $x\in B$. If $x\in A$ then $\de(a,x)\leq r$, hence $\de(b,\alpha(x))\leq r$, hence $\alpha(x)\in B$,  hence $\restrict{\mathrm{M}}{A}$ is isometric to $\restrict{\mathrm{M}}{B}$.   Let $\mathfrak{p}$ be a \Kat function of $\mathrm{M}$ with $\dom(\mathfrak{p})\subseteq A$ and $\mathfrak{p}(x)\leq r$ for all $x\in \dom(\mathfrak{p})$. Then $\orb(\mathfrak{p})\subseteq A$. Hence $\restrict{\mathrm{M}}{A}$ is a homogeneous metric space with $[0,r]\cap \mathcal{D}$ as set of distances. For $x\in M$ let $[x]_r$ denote the $\stackrel{r}{\sim}$ equivalence class containing $x$.

\begin{defin}\label{defin:distequ}
For $A,B\in M{/\negthickspace\stackrel{r}{\sim}}$ let \\$\de(A,B)=\{\de(a,b): \text{$a\in A$ and $b\in B$}\}$ and let \\$\de_{\min}(A,B)=\min\de(A,B)$ and $\de_{\max}(A,B)=\max\de(A,B)$. 
\end{defin}

\begin{lem}\label{lem:pairdist}
Let   $A,B\in M{/\negthickspace\stackrel{r}{\sim}}$ with $A\not=B$ and $a\in A$ and $n\in \de(A,B)$, then:
\begin{enumerate}
\item $n>2r$.
\item There is $b\in B$ with $\de(a,b)=n$.
\item  $\de_{\max}(A,B)-\de_{\min}(A,B)\leq r$.
\end{enumerate}
\end{lem}
\begin{proof}
Let $a\in A$ and $b\in B$ with $\de(a,b)=n\leq 2r$ and $\mathfrak{k}$ the \Kat function with $\dom(\mathfrak{k})=\{a,b\}$ and $\mathfrak{k}(a)=r$ and $\mathfrak{k}(b)=r$. Let $c\in \orb(\mathfrak{k})$. Then $c\in A$ and $c\in B$, hence $A=B$. 

Let $a'\in A$ and $b'\in B$ with $\de(a',b')=n$ and $\mathfrak{k}$ the type function with $\dom(\mathfrak{k})=\{a',b',a\}$ and $\mathfrak{k}(a)=n$ and $\mathfrak{k}(a')=\de(a,b')$ and $\mathfrak{k}(b')=\de(a,a')$. Because $\Sp(\mathfrak{k})$ is metric the type function $\mathfrak{k}$ is a \Kat function and hence there is an element $b\in \orb(\mathfrak{k})$. Because $\de(b,b')<r$ the point $b$ is an element of $B$ and $\de(a,b)=\mathfrak{k}(a)=n$.

Let $a\in A$ and $b\in B$ with $\de(a,b)=\de_{\max}(A,B)$. There is $c\in B$ with $\de(a,c)=\de_{\min}(A,B)$. The triangle $a,b,c$ is metric with $\de(b,c)\leq r$. 
\end{proof}

\begin{lem}\label{lem:trdist}
The triangle $A,B,C$ is metric for all $A,B,C\in  M{/\negthickspace \stackrel{r}{\sim}}$ under the distance function $\de_{\min}$. 
\end{lem}
\begin{proof}
Let $a\in A$. According to Lemma \ref{lem:pairdist} there exists $b\in B$ and $c\in C$ with $\de(a,b)=\de_{\min}(A,B)$ and $\de(a,c)=\de_{\min}(A,C)$. Hence $\de(a,b)+\de(a,c)\geq \de(b,c)\geq \de_{\min}(B,C)$. 
\end{proof}

\begin{lem}\label{lem:embedmin}
Let $(A_i; i\in n\in \omega)$ be pairwise different $\stackrel{r}{\sim}$-equivalence classes. Then there exist points $(a_i; i\in n)$ with $a_i\in A_i$ and $\de(a_i,a_j)=\de_{\min}(A_i,A_j)$ for all $i,j\in n$. 
\end{lem}
\begin{proof}
By induction on $n$. Let $(A_i; i\in n\in \omega)$ be pairwise different $\stackrel{r}{\sim}$-equivalence classes  and  $(a_i; i\in n)$ points with $a_i\in A_i$ and $\de(a_i,a_j)=\de_{\min}(A_i,A_j)$ for all $i,j\in n$. Let $A$ be an $\stackrel{r}{\sim}$-equivalence class different from the $A_i$. Let $b\in A$ and $\mathfrak{p}$ the type function with $\dom(\mathfrak{p})=\{a_i : i\in n\}\cup \{b\}$ and $\mathfrak{p}(a_i)=\de_{\min}(A_i,A)$ and $\mathfrak{p}(b)=r$. Using inequalities~(1) of Section 2 it follows from Lemma \ref{lem:trdist} that it suffices to check the inequalities $|\mathfrak{p}(a_i)-r|\leq \de(a_i,a)\leq \mathfrak{p}(a_i)+r$, in order to verify that $\mathfrak{p}$ is a \Kat function.  That is the inequalities $\de_{\min}(A_i,A)-r\leq \de(a_i,a)\leq \de_{\min}(A_i,A)+r$ which follow from Lemma \ref{lem:pairdist} Item (3).

\end{proof}

\section{Distance beteween orbits}

Let $\mathcal{D} $ be a universal set of numbers  and $\mathcal{B}$ a block of $\mathcal{D}$ with $\min\mathcal{B}=\mathbf{m}$.
Let $\mathrm{M}=(M, \de)\in \mathfrak{U}_{\mathcal{D}}$.

\begin{defin}\label{defin:Katdist}
Let   $\mathfrak{s}$ and $\mathfrak{t}$ be two \Kat functions of\/ $\mathrm{M}$.  Then 
\begin{align*}
&\de(\mathfrak{s},\mathfrak{t})=\{m\in \mathcal{D} : \exists x\in \orb(\mathfrak{s})\, \exists y\in \orb(\mathfrak{t}) \big(\de(x,y)=m)\big)\},\\
&\de_{\mathrm{min}}(\mathfrak{s},\mathfrak{t})=\min(\de(\mathfrak{s}, \mathfrak{t})),\\
&\de_{\mathrm{max}}(\mathfrak{s},\mathfrak{t})=\max(\de(\mathfrak{s}, \mathfrak{t})).
\end{align*} 
\end{defin}

\begin{lem}\label{lem:range}
Let  $A$ be a finite subset of $M$ and $\mathfrak{s}$ and $\mathfrak{t}$ two \Kat functions with $\dom(\mathfrak{s})=\dom(\mathfrak{t})=A$. Then 
\begin{align*}
&\de(\mathfrak{s}, \mathfrak{t})=\\
&\big\{m\in \mathcal{D} : \max\{|\mathfrak{s}(x)-\mathfrak{t}(x)| : x\in A\}\leq m\leq \min\{\mathfrak{s}(x)+\mathfrak{t}(x) : x\in A\}\big\}.
\end{align*}
\end{lem}
\begin{proof}
Let $v\in \orb(\mathfrak{s})$.  The type function   $\mathfrak{t}'$   with $\dom(\mathfrak{t}')=A\cup \{v\}$ and $\mathfrak{t}\subseteq \mathfrak{t'}$ and $\mathfrak{t}'(v)=m$ is a \Kat function if and only if $m\in \de(\mathfrak{s}, \mathfrak{t})$. Hence, in order to check that $\mathfrak{t}'$ is a \Kat function, we have to verify that every triangle $x,v,\mathfrak{t}'$ with $x\in A$ is metric, that is if $|\de(x,\mathfrak{t}')-\de(x,v)|\leq m \leq \de(x,\mathfrak{t}')+\de(x,v)$.

\end{proof}

\begin{cor}\label{cor:range}
Let  $A$ be a finite subset of $M$ and $\mathfrak{s}$ and $\mathfrak{t}$ two \Kat functions with $\dom(\mathfrak{s})=\dom(\mathfrak{t})=A$.Then:
\begin{enumerate}
\item $\de_{\mathrm{min}}(\mathfrak{s}, \mathfrak{t})= \min\big\{n\in \mathcal{D}: n\geq \max\{|\mathfrak{s}(x)-\mathfrak{t}(x)|: x\in A\}\big\}$. 

\item $\de_{\mathrm{max}}(\mathfrak{s},\mathfrak{t})= \max\big\{n\in \mathcal{D}: n\leq\min\{\mathfrak{s}(x)+\mathfrak{t}(x) : x\in A\}\big\}$. 

\item Let $r\in \mathcal{B}$  with $\mathbf{m}\leq r^{\langle -\rangle}<r$ and $\rank(\mathfrak{s})\geq r^{\langle -\rangle}$ and $\rank(\mathfrak{t})\geq \mathbf{m}$. Then $\de_{\max}(\mathfrak{s},\mathfrak{t})\geq r$.

\item $\de_{\max}(\mathfrak{s},\mathfrak{s})=\max\big\{n\in \mathcal{D}: n\leq 2\cdot \rank(\mathfrak{s})\big\}$. (Hence $\de_{\max}(\mathfrak{s}, \mathfrak{s})=\de_{\max}(\mathfrak{t}, \mathfrak{t})$ iff $\rank(\mathfrak{s})=\rank(\mathfrak{t})$.)

\item $\de(\mathfrak{s}, \mathfrak{s})=\big\{n\in \mathcal{D}: 0\leq n\leq 2\cdot \rank(\mathfrak{s})\big\}$.
\end{enumerate}

\end{cor}
\begin{proof}
Item 3.:  For $\rank(\mathfrak{s})\geq r$ Item 3. follows because  Item 2. implies $\de_{\max}(\mathfrak{s},\mathfrak{t})\geq \rank(\mathfrak{s})$. Let $\rank(\mathfrak{s})=r^{\langle -\rangle}$. Then for all $x\in A$:   $\mathfrak{s}(x)+\mathfrak{t}(x)\geq \rank(\mathfrak{s})+\mathbf{m}\geq r^{\langle -\rangle}+\mathbf{m}\geq r$. ($ r^{\langle -\rangle}+\mathbf{m}\geq r$ follows from Theorem~\ref{thm:distset}.) 
\end{proof}

\begin{lem}\label{lem:Katsp}
Let $A$ be a finite subset of $M$ and $s\in \omega$ and $\mathfrak{t}_i$ a \Kat function with $\dom(\mathfrak{t}_i)=A$  for every $i\in s$. (The \Kat functions $\mathfrak{t}_i$ need not be pairwise different.)  Let $\de''$  be a metric on the index set $s$   so that $\de''(i,j)\in \de(\mathfrak{t}_i,\mathfrak{t}_j)$ for all $i,j\in s$. Then the distance function $\de'$ on $A\cup s$ given by
\begin{enumerate}
\item $\de'(x,y)=\de(x,y)$ for all $x,y\in A$ and $\de'(i,j)=\de''(i,j)$ for all $i,j\in s$, 
\item $\de'(x,i)=\mathfrak{t}_i(x)$ for all $x\in A$ and $i\in s$,
\end{enumerate}
is  metric and every partial isometry $\beta$ of $(A\cup C,\de')$ for $C\subseteq s$  into $\mathrm{M}$  with $\beta(x)=x$ for all $x\in A$ has an extension to an isometry $\alpha$ of $(A\cup s;\de')$ into $\mathrm{M}$  with $\alpha(i)\in \orb(\mathfrak{t}_i)$ for all $i\in s$.  
\end{lem} 
\begin{proof}
Because $\de''$ is a metric on $s$ and $\de$ is a metric on $A$ and every $\mathfrak{t}_i$ is a \Kat function it remains  to check that the triangles of the form $x,i,j$ with $x\in A$ and $i,j\in s$ are metric in the distance function $\de'$.  Which indeed is the case because $|\de'(x,i)-\de'(x,j)|=|\mathfrak{t}_i(x)-\mathfrak{t}_j(x)|\leq \de_{\min}(\mathfrak{t}_i,\mathfrak{t}_j)\leq   \de''(i,j)=\de'(i,j)=\de''(i,j)\leq \de_{\max}(\mathfrak{t}_i,\mathfrak{t}_j)\leq \mathfrak{t}_i(x)+\mathfrak{t}_j(x)=\de'(x,i)+\de'(x,j)$.

It follows from the mapping extension property that every partial isometry $\beta$ of $A\cup C$ for $C\subseteq s$  into $\mathrm{M}$  with $\beta(x)=x$ for all $x\in A$ has an extension to an isometry $\alpha$ of $(A\cup s;\de')$ into $\mathrm{M}$. Item 2. implies that $\alpha(i)\in \orb(\mathfrak{t}_i)$.
\end{proof}

\begin{cor}\label{cor:Katsp} 
Given the conditions of Lemma \ref{lem:Katsp}: For every $k\in s$ and $v\in \orb(\mathfrak{t}_k)$  there exists a set of points $\{w_i: i\in s\}$  so that $w_k=v$ and  $w_i\in \orb(\mathfrak{t}_i) $ and $\de(w_i,w_j)=\de''(i,j)$ for all $i,j\in s$. 
\end{cor}

\begin{lem}\label{lem:shrinkfin}
Let   $A,B, R$ be  finite disjoint subsets of $M$ and   $\mathfrak{T}$ a set of \Kat functions $\mathfrak{t}$ with $\dom(\mathfrak{t})=A$. For every $\mathfrak{t}\in \mathfrak{T}$ let $\mathfrak{t}'$ be a \Kat function with $\mathfrak{t}\subseteq \mathfrak{t}'$ and $\dom(\mathfrak{t}')=A\cup B$. Also: $\de(\mathfrak{t}, \mathfrak{s})=\de(\mathfrak{t}', \mathfrak{s}')$ for all $\mathfrak{t}, \mathfrak{s}\in \mathfrak{T}$. (Implying $\rank(\mathfrak{t})=\rank(\mathfrak{t}')$ for all $\mathfrak{t}\in \mathfrak{T}$.)

Then there exists an embedding $\alpha$ of $A\cup R$ into $\mathrm{M}$ so that $\alpha(x)=x$ for all $x\in A$ and $\alpha(y)\in \orb(\mathfrak{t'})$  for all $\mathfrak{t}\in \mathfrak{T}$ and all $y\in \orb(\mathfrak{t})\cap R$.
\end{lem}
\begin{proof}
Let $S=\{x\in R: \exists \mathfrak{t}\in \mathfrak{T}\, \big(x\in \orb(\mathfrak{t})\big)\}$. Then $x\not\in \orb(\mathfrak{s})$,  if $x\in \orb(\mathfrak{t})$ and $\mathfrak{t}\not=\mathfrak{s}$. Let $(s_i; i\in n)$ be an enumeration of the elements of $S$. For every $i\in n$ let $\mathfrak{t}_i\in \mathfrak{T}$ be such that $s_i\in \orb(\mathfrak{t}_i)$. (Which implies that a $\mathfrak{t}\in \mathfrak{T}$ might be enumerated several times.) Let $\de''$ be the metric on $n$ with $\de''(i,j)=\de(s_i,s_j)$. Then $\de''(i,j)\in\de(\mathfrak{t_i}, \mathfrak{t_j})= \de(\mathfrak{t_i}', \mathfrak{t_j}')$   and hence it follows from Lemma  \ref{lem:Katsp} that there exists an embedding $\beta$ of $A\cup B\cup  n$ into $\mathrm{M}$ with $\beta(x)=x$ for all $x\in A\cup B$ and $\beta(i)\in \mathfrak{t}_i'$ for all $i\in n$. This in turn implies that there is an embedding $\gamma$ of $A\cup B\cup S $ into $\mathrm{M}$ with $\gamma(x)=x$ for all $x\in A\cup B$ and $\gamma(s_i)\in \mathfrak{t}_i'$ for all $i\in n$.

Let $\alpha$ be the extension of $\gamma$ to $A\cup B\cup R$.
\end{proof}

\section{The orbit  amalgamation theorem}

Let $\mathcal{D}$ be a universal set of numbers and $\mathrm{M}=(M, \de)\in\mathfrak{U}_{\mathcal{D}}$. Let $E=(v_k; k\in \omega)$ be an $\omega$-enumeration of  $M$. For $m\in \omega$ we denote by $E_m=(v_k;k\in m)$ the initial interval of $E$ of length $m$.

\begin{lem}\label{lem:reduce}
Let    $l\in m\in \omega$ and $s\in \omega$. For all  $i,j\in s$ let $\mathfrak{t}_i\subseteq \mathfrak{s}_i$   be  \Kat functions with $\dom(\mathfrak{t}_i)=E_l$ and $\dom(\mathfrak{s}_i)=E_m$ and   $\de(\mathfrak{t}_i,\mathfrak{t}_j)=\de(\mathfrak{s}_i, \mathfrak{s}_j)$. (Hence $\rank(\mathfrak{t}_i)=\rank(\mathfrak{s}_i)$).     

Then there exists an embedding $\alpha$ of  $\mathrm{M}$ into $\mathrm{M}$  with $\alpha(x)=x$ for all $x\in E_l$ and $\alpha(x)\in \orb(\mathfrak{s}_i)$ for all $x\in \orb(\mathfrak{t}_i)$.  
\end{lem}
\begin{proof}
Let $M'=(M\setminus E_m)\cup E_l$.  According to Lemma \ref{lem:finiterem} it suffices to prove that there exists an embedding $\alpha$ of  $\restrict{\mathrm{M}}{M'}$ into $\restrict{\mathrm{M}}{M'}$  with $\alpha(x)=x$ for all $x\in E_l$ and $\alpha(x)\in \orb(\mathfrak{s}_i)$ for all $x\in \orb(\mathfrak{t}_i)\cap M'$.

For $m\leq h\in \omega$ let  $\mathcal{A}_h$ be the set of isometries $\beta$ with $\beta(v_k)=v_k$ for all $k\in l$ and  $\dom(\beta)=(E_h\setminus E_m)\cup E_l$      and  $\beta(x)\in \orb(\mathfrak{s}_i)$ for all $x\in \orb(\mathfrak{t}_i)$ and $i\in s$. It follows from Lemma \ref{lem:shrinkfin} that $\mathcal{A}_j$ is not empty. Let $\mathcal{A}=\bigcup_{m\leq h\in \omega}\mathcal{A}_h$.  For two isometries $\beta$ and $\gamma$ in $\mathcal{A}$ let $\beta\preceq \gamma$ if $\beta\in \mathcal{A}_j$ and $\gamma\in \mathcal{A}_h$ with $j\leq h$ and if
$\de(y,\beta(x))=\de(y,\gamma(x))$ for all $y\in E_m\setminus E_l$ and all $x\in \dom(\beta)$. Let $\beta\sim \gamma$ if $\beta\preceq \gamma$ and $\gamma\preceq\beta$.  Then $\sim$ is an equivalence relation and because $\mathcal{D}$ is finite there are only finitely many $\sim$ equivalence classes in every $\mathcal{A}_h$. The quasiorder $\preceq$ on $\mathcal{A}$ factors into a partial order $\mathfrak{P}$ on the $\sim$ equivalence classes. The restriction of an isometry $\beta\in \mathcal{A}_{h+1}$ to $E_h$ is an isometry in $\mathcal{A}_h$.

According to   K\H{o}nig's Lemma there exists a chain $(\mathcal{C}_{m+j}; j\in \omega)$ in $\mathfrak{P}$. ($\mathcal{C}_m$ being the equivalence class containing the empty function.) The Theorem will follow if for every $h$ with  $m\leq h\in\omega$ and $\beta\in \mathcal{C}_h$ there is a $\gamma\in \mathcal{C}_{h+1}$ with $\beta\preceq  \gamma$. 

Let $\beta\in \mathcal{C}_h $ and $\delta\in \mathcal{C}_{h+1}$. 
 Let  $\kappa$ be the  isometry  with $\dom(\kappa)=\{\delta(v_k) : k\in h\}$ for which $\kappa(\delta(v_k))=\beta(v_k)$ for all $k\in h\setminus m$ and $\kappa(v_k)=v_k$ for all $i\in m$.
Let $\mu$ be an  extension of  $\kappa$ to an isometry of $\{\delta(v_k): k\in h+1\}$ into $\mathrm{M}$. That is the domain of $\mu$ contains the element $v_h$ in   addition to the elements in the domain of $\kappa$. Such an extension exists because $\mathrm{M}$ is homogeneous.   Then $\beta\preceq\gamma:=\beta\cup \{v_h, \mu(\delta(v_h))\}$. It remains to show that $\gamma\in \mathcal{C}_{h+1}$.

Let $i\in s$ and  $v_h\in \orb(\mathfrak{t}_i)$. Then $\delta(v_h)\in \orb(\mathfrak{s}_i)$ because $\delta\in \mathcal{C}_{h+1}\subseteq \mathcal{A}$ implying that $ \mu(\delta(v_h))\in \orb(\mathfrak{s}_i)$ because  $\mu(x)=x$ for all $x\in E_m=\dom(\mathfrak{s}_i)$. It follows that $\gamma\in \mathcal{C}_{h+1}$. 
\end{proof}

\begin{thm}\label{thm:reducegen}
Let  $A$ and $ B$ be   finite subsets of $M$ with $A\cap B=\emptyset$ and $s\in \omega$ and for every $i\in s$  let $\mathfrak{t}_i\subseteq \mathfrak{s}_i$ be  \Kat functions  with $\dom(\mathfrak{t}_i)=A$ and $\dom(\mathfrak{s}_i)=A \cup B$  and $\de(\mathfrak{t}_i, \mathfrak{t}_j)=\de(\mathfrak{s}_i, \mathfrak{s}_j)$ for all $i,j\in s$. 

Then there exists an embedding $\alpha$ of\/ $\mathrm{M}$ into $\mathrm{M}$  with $\alpha(x)=x$ for all $x\in A$ and $\alpha(x)\in \orb(\mathfrak{s}_i)$ for all $x\in \orb(\mathfrak{t}_i)$ and all $i\in s$; that is with $\orb_{\alpha(\mathrm{M})}(\mathfrak{t}_i)\subseteq \orb_{\mathrm{M}}(\mathfrak{s}_i)$.  
\end{thm}
\begin{proof}
The Theorem follows from Lemma \ref{lem:reduce} for $E=(v_i; i\in \omega)$  an enumeration of $M$ with $E_l=A$ and $E_m=B$.
\end{proof}

\begin{cor}\label{cor:reduce}
Let   $\mathfrak{t}\subseteq \mathfrak{s}$ be two \Kat functions of\/ $\mathrm{M}$ with $\rank(\mathfrak{t})=\rank(\mathfrak{s})$.  

Then there exists an embedding $\alpha$ of\/ $\mathrm{M}$ into $\mathrm{M}$  with $\alpha(x)=x$ for all $x\in\dom(\mathfrak{t})$ and $\alpha(x)\in \orb(\mathfrak{s})$ for all $x\in \orb(\mathfrak{t})$, that is with $\orb_{\alpha(\mathrm{M})}(\mathfrak{t})\subseteq \orb_{\mathrm{M}}(\mathfrak{s})$.  
\end{cor}

\section{The first block of $\mathcal{D}$ and  orbit distances}

Let $\mathcal{D}$ be a universal set of numbers  and $\mathcal{B}$ the first block of $\mathcal{D}$, that is $1:=\min\mathcal{B}=\min(\mathcal{D}\setminus\{0\})$.   Let $\mathrm{M}=(M, \de)\in \mathfrak{U}_{\mathcal{D}}$.

\begin{lem}\label{lem:minmet}
Let $A$ be a finite subset of $M$ and $\mathfrak{p}_0$, $\mathfrak{p}_1$, $\mathfrak{p}_2$ three \Kat function with $\dom(\mathfrak{p}_i)=A$ for $i\in 3$. Then the trianlge $\mathfrak{p}_0, \mathfrak{p}_1, \mathfrak{p}_2$ with distance function $\de_{\mathrm{min}}$ is metric. 
\end{lem}
\begin{proof}
Assume for  a contradiction that 
\[
\de_{\mathrm{min}}(\mathfrak{p}_0, \mathfrak{p}_1)+\de_{\mathrm{min}}(\mathfrak{p}_0, \mathfrak{p}_2)< \de_{\mathrm{min}}(\mathfrak{p}_1, \mathfrak{p}_2).
\]
Let $w_0\in \orb(\mathfrak{p}_0)$. There exist points $w_1\in \orb(\mathfrak{p}_1)$ and $w_2\in \orb(\mathfrak{p}_2)$ with $\de(w_0, w_1)=\de_{\mathrm{min}}(\mathfrak{p}_0, \mathfrak{p}_1)$ and  $\de(w_0, w_2)=\de_{\mathrm{min}}(\mathfrak{p}_0, \mathfrak{p}_2)$. Then
\begin{align*}
&\de(w_0, w_1)+\de(w_0, w_2)\geq \de(w_1, w_2)\geq \de_{\mathrm{min}}(\mathfrak{p}_1, \mathfrak{p}_2)>\\
&\de_{\mathrm{min}}(\mathfrak{p}_0, \mathfrak{p}_1)+\de_{\mathrm{min}}(\mathfrak{p}_0, \mathfrak{p}_2)=\de(w_0, w_1)+\de(w_0, w_2).
\end{align*}
\end{proof}

\begin{defin}\label{defin:minmet}
Let $A$ be a finite subset of $M$ and $s\in \omega$ and for every $i\in s$ let $\mathfrak{p}_i$ be a \Kat function with $\dom(\mathfrak{p}_i)=A$. Let $r\in \mathcal{B}$ with $1\leq r^{\langle -\rangle}<r$.   Then,    the $\mathcal{D}$-graph with set of points $\{0, 1, 2, \dots, s-1\}=s$ and distance function $\de'$ so that for all $i,j\in s$ :  
\[
\de'(i,j)=
\begin{cases}
0                                                                                 &\text{if $i=j$;}\\
\de_{\min}(\mathfrak{p}_i, \mathfrak{p}_j),                 &\text{if $\de_{\min}(\mathfrak{p}_i, \mathfrak{p}_j)\geq r$;}\\
\in \{r^{\langle-\rangle}, r\},                                           &\text{otherwise,}
\end{cases}
\]
is  an {\em $r$-levelling distance function on the indices of $(\mathfrak{p}_i; i\in s)$.}

\end{defin}

\begin{lem}\label{lem:minmetex}
Let $A$ be a finite subset of $M$ and $s\in \omega$ and for every $i\in s$ let $\mathfrak{p}_i$ be a \Kat function with $\dom(\mathfrak{p}_i)=A$. Then the $\mathcal{D}$-graph with set of points $\{\mathfrak{p}_i: i\in s\}$ and distance function $\de_{\mathrm{min}}$ is  metric. 

Let $r\in \mathcal{B}$ with $1<r$. Then every    $r$-levelling distance function $\de'$ on the set $s$ of indices of $(\mathfrak{p}_i; i\in s)$ is a metric space 
 so that for $i\not=k\not=j$: 
\begin{align*}
&|\de'(k,i)-\de'(k,j)|\leq \de_{\min}(\mathfrak{p}_i, \mathfrak{p}_j) \text{\  if\/  $\mathfrak{p}_i\not=\mathfrak{p}_j$ and  }\\
&|\de'(k,i)-\de'(k,j)|\leq 1 \text{\  if\/ $\mathfrak{p}_i=\mathfrak{p}_j$.  }\notag
\end{align*} 
\end{lem}
\begin{proof}
It follows from Lemma \ref{lem:minmet} that the $\mathcal{D}$-graph with set of points $\{\mathfrak{p}_i: i\in s\}$ and distance function $\de_{\mathrm{min}}$ is a metric space.  

The $\mathcal{D}$-graph with set of points $s$ and distance function $\de'$ is a metric space because all  triangles  $i,j,k$ are metric in the distance function  $\de'$ as can be easily verified.

Let $\mathfrak{p}_i=\mathfrak{p}_j$. Then $\de_{\min}(\mathfrak{p}_k, \mathfrak{p}_i)=\de_{\min}(\mathfrak{p}_k, \mathfrak{p}_j)$. If $\de_{\min}(\mathfrak{p}_k,\mathfrak{p}_i)\geq r$ then $\de'(k,i)=\de'(k,j)$ and hence $|\de'(k,i)=\de'(k,j)|=0<1$.  If $\de_{\min}(\mathfrak{p}_k,\mathfrak{p}_i)< r$  then  $|\de'(k,i)-\de'(k,j)|\leq  |r-r^{\langle -\rangle}|\leq 1$ according to Theorem \ref{thm:distset}. 
   
Let $\mathfrak{p}_i\not=\mathfrak{p}_j$.  If  both $\de_{\min}(\mathfrak{p}_k, \mathfrak{p}_i)< r$ and $\de_{\min}(\mathfrak{p}_k, \mathfrak{p}_j)< r$ then $|\de'(k,i)-\de'(k,j)|\leq |r-r^{\langle -\rangle}|\leq 1\leq \de_{\min}(\mathfrak{p}_i, \mathfrak{p}_j)$.  If   $\de_{\min}(\mathfrak{p}_k, \mathfrak{p}_i)\geq r$ and $\de_{\min}(\mathfrak{p}_k, \mathfrak{p}_j)< r$  then $|\de'(k,i)-\de'(k,j)|\leq |\de_{\min}(\mathfrak{p}_k, \mathfrak{p}_i)-r^{\langle -\rangle}|\leq |\de_{\min}(\mathfrak{p}_k,\mathfrak{p}_i)-\de_{\min}(\mathfrak{p}_k,\mathfrak{p}_j)|\leq \de_{\min}(\mathfrak{p}_i, \mathfrak{p}_j) $; the last inequality following from Lemma \ref{lem:minmet}.

If  both $\de_{\min}(\mathfrak{p}_k, \mathfrak{p}_i)\geq  r$ and $\de_{\min}(\mathfrak{p}_k, \mathfrak{p}_j)\geq  r$ then      $|\de'(k,i)-\de'(k,j)|= |\de_{\min}(\mathfrak{p}_k,\mathfrak{p}_i)-\de_{\min}(\mathfrak{p}_k,\mathfrak{p}_j)|\leq \de_{\min}(\mathfrak{p}_i, \mathfrak{p}_j) $.             
\end{proof}

\begin{cor}\label{cor:minmetr}
Let $A$ be a finite subset of $M$ and $ s\in \omega$ and for every $i\in s+1$ let $\mathfrak{p}_i$ be a \Kat function with $\dom(\mathfrak{p}_i)=A$ and $\mathfrak{p}_i\not=\mathfrak{p}_j$ for $i\not= j$ and $i,j\in s$. Let $t\in s$ and $\mathfrak{p}_t=\mathfrak{p}_s$ and $v\in \orb(\mathfrak{p}_s)$.
Let $1\leq r^{\langle -\rangle}<r\in \mathcal{B}$ and   $\rank(\mathfrak{p}_i)\in \{r,r^{\langle - \rangle}\}$ for all $1\leq i\in s+1$.  

Then for every  $r$-levelling distance function $\de'$ on the indices of $(\mathfrak{p}_i;  i\in s+1)$  there exists a set $\{w_i\in \orb(\mathfrak{p}_i) : i\in s+1\}$ of points with $w_s=v$ and $\de(w_i,w_j)=\de'(i,j)$. Also: 
\[
|\de(w_k,w_i)-\de(w_k,w_j)|\leq \de_{\min}(\mathfrak{p}_i, \mathfrak{p}_j) \text{\  for  $i,j,k\in s$ with $i\not=j\not=k\not=i$.}
\]
\end{cor}
\begin{proof}
If $\de_{\min}(\mathfrak{p}_i,\mathfrak{p}_j)\geq r$ then $\de'(i,j)=\de_{\min}(\mathfrak{p}_i,\mathfrak{p}_j)\in \de(\mathfrak{p}_i,\mathfrak{p}_j)$.   If $\de_{\min}(\mathfrak{p}_i,\mathfrak{p}_j)\\< r$, it follows from $\rank(\mathfrak{p}_i)\in \{r^{\langle -\rangle}, r\}$ or $\rank(\mathfrak{p}_j)\in \{r^{\langle -\rangle}, r\}$  according to  Corollary~\ref{cor:range} Item 3.  that $ \de_{\max}(\mathfrak{p}_i,\mathfrak{p}_j)\geq r$ and therefore  $\de'(i,j)\in \de(\mathfrak{p}_i, \mathfrak{p}_j)$.  Hence the Corollary     follows from Lemma \ref{lem:minmetex} and  Corollary \ref{cor:Katsp}. 
\end{proof}

\section{The central extension theorem}

Let $\mathcal{D}$ be a universal set of numbers  and $\mathcal{B}$ the first block of $\mathcal{D}$, that is $1:=\min\mathcal{B}=\min(\mathcal{D}\setminus\{0\})$.   Let $\mathrm{M}=(M, \de)\in \mathfrak{U}_{\mathcal{D}}$.

\begin{defin}\label{defin:extend}
Let $S\subseteq M$ and $1<r\in \mathcal{B}$. A \Kat function $\mathfrak{p}$ of\/ $\mathrm{M}$ with $\rank(\mathfrak{p})=r$  is {\em extendible into $S$} on $\mathrm{M}$  if for every copy $\mathrm{H}=(H;\de)$ of\/ $\mathrm{M}$ in $\mathrm{M}$ with $\dom(\mathfrak{p})\subseteq H$  and every \Kat function $\mathfrak{p}'$ with $\dom(\mathfrak{p}')\subseteq H$ and $\mathfrak{p}\subseteq \mathfrak{p}'$  and $\rank(\mathfrak{p}')=r$ there exists an embedding $\alpha$ of\/ $\mathrm{H}$ into $\mathrm{H}$ with $\alpha(x)=x$ for all $x\in \dom(\mathfrak{p}')$ and  a \Kat function $\mathfrak{g}$ with  $\dom(\mathfrak{g})\subseteq \alpha(H)$ and with $\mathfrak{p}'\subseteq \mathfrak{g}$ and $\rank(\mathfrak{g})=r^{\langle -\rangle}$ so that $\orb_{\alpha(\mathrm{H})}(\mathfrak{g})\subseteq S$. 
\end{defin}

Note that if a \Kat function $\mathfrak{p}$ with $\rank(\mathfrak{p})=r$ is extendible into $S$ on $\mathrm{M}$ and if $\mathrm{H}$ is a copy of\/ $\mathrm{M}$ in $\mathrm{M}$  and $\mathfrak{q}$ is a \Kat function of\/ $\mathrm{H}$ with $\mathfrak{p}\subseteq \mathfrak{q}$ and $\dom(\mathfrak{q})\subseteq H$  and $\rank(\mathfrak{q})=r$, then $\mathfrak{q}$ is extendible into $S$ on $\mathrm{H}$.

\begin{lem}\label{lem:germ}
Let $1\leq r^{\langle -\rangle}<r\in \mathcal{B}$ and let $A$ be a finite subset of $M$ and $S\subseteq M$  and  $ t\in s\in \omega$     and $\mathfrak{p}_i$   a  \Kat function  for every $i\in s$ so that:  
\begin{description}
\item[\textnormal{1}.]  $\dom(\mathfrak{p}_i)=A$  for all $i\in s$ and $\mathfrak{p}_i\not=\mathfrak{p}_j$ for $i\not=j$.  
\item[\textnormal{2}.]  $\rank(\mathfrak{p}_i)=r$ for  $1\leq i\in t+1$ and $\rank(\mathfrak{p}_i)=r^{\langle -\rangle}$ for  $t<i\in s$.
\item[\textnormal{3}.]  $\mathfrak{p}_i$ for $1\leq i\leq t$ is extendible into $S$ and if $t=0$ then $\rank(\mathfrak{p}_0)=r$  and $\mathfrak{p}_0$ is extendible into $S$.
\end{description}

Then there exists an embedding $\alpha$ of\/ $\mathrm{M}$ into $\mathrm{M}$ with $\alpha(x)=x$ for all $x\in A$ and  with image $\mathrm{H}=(H; \de)$  and a point $v\in H$ and for every $i\in s$ a \Kat function $\mathfrak{p}_i'$ so that:

\begin{enumerate}
\item $\dom(\mathfrak{p}_i')=A\cup \{v\}$ and $\mathfrak{p}_i\subseteq \mathfrak{p}_i'$ for all $i\in s$ and $\mathfrak{p}_i'\not=\mathfrak{p}_j'$ for $i\not=j$.
\item $\rank(\mathfrak{p}_i')=\rank(\mathfrak{p}_i)$ for all $1\leq i\in s$ with $i\not=t$ and $\rank(\mathfrak{p}_t')=r^{\langle -\rangle}$.
\item $\mathfrak{p}_i'$ for $1\leq i < t$ is extendible into $S$ and if\/   $\mathfrak{p}_0$ is extendible into $S$ and $t\geq 1$ then $\mathfrak{p}'_0$ is extendible into $S$.
\item $\orb_{\mathrm{H}}(\mathfrak{p}'_t)\subseteq S$.
\item $\de_{\min}(\mathfrak{p}_i', \mathfrak{p}_j')=\de_{\min}(\mathfrak{p}_i, \mathfrak{p}_j)$ for all $i,j\in s$ with $i\not=j$.
\end{enumerate} 
\end{lem}
\begin{proof}
Because $\mathfrak{p}_t$ is extendible into  $S$ there exists an embedding $\alpha'$ of $\mathrm{M}$ into $\mathrm{M}$ with $\alpha'(x)=x$ for all $x\in A$ with image $\mathrm{M}'=(M', \de)$  and a \Kat function $\mathfrak{g}$ with $\mathfrak{p}_t\subseteq \mathfrak{g}$ and $\rank(\mathfrak{g})=r^{\langle -\rangle}$ and $\orb_{\mathrm{M}'}(\mathfrak{g})\subseteq S$.

Let $v\in \orb_{\mathrm{M}'}(\mathfrak{g})$ and $\mathfrak{g}''$ the \Kat function with $\dom(\mathfrak{g}'')=\dom(\mathfrak{g})\cup \{v\}$ and $\mathfrak{g}\subseteq \mathfrak{g}''$ and $\mathfrak{g}''(v)=r^{\langle -\rangle}$. It follows from Theorem \ref{thm:orbits} that $\mathfrak{g}''$ is a \Kat function.   Let $\mathfrak{g}'$ be the \Kat function with $\dom(\mathfrak{g}')=A\cup \{v\}$ and $\mathfrak{g}'\subseteq \mathfrak{g}''$. Then $\rank(\mathfrak{g}'')=\rank(\mathfrak{g}')=r^{\langle -\rangle}$.  It follows from Corollary~\ref{cor:reduce} with $\mathfrak{g}'$ for $\mathfrak{t}$ and with $\mathfrak{g}''$ for $\mathfrak{s}$ that there exists an isometry $\alpha''$ of\/ $\mathrm{M}'$ onto $\mathrm{M}'$ with $\alpha''(x)=x$ for all $x\in A\cup \{v\}$ and $\alpha''(x)\in \orb_{\mathrm{M}'}(\mathfrak{g}'')\subseteq \orb_{\mathrm{M}'}(\mathfrak{g})$ for all $x\in \orb_{\mathrm{M}'}(\mathfrak{g}')$. Let $\mathrm{H}=(H;\de)$ be the image of $\alpha''$ and $\alpha=\alpha''\circ\alpha'$. Note that $\orb_{\mathrm{H}}(\mathfrak{g}')\subseteq \orb_{\mathrm{M}'}(\mathfrak{g})\subseteq S$ and that $v\in H$.

Let $\mathfrak{p}_s$ be the \Kat function with $\mathfrak{p}_s=\mathfrak{p}_t$ and let $\de'$ be the $r$-levelling metric on $s+1$ given by:
\begin{description}
\item[\textnormal{1}.] $\de'(s,t)=r^{\langle -\rangle}$.
\item[\textnormal{2}.] $\de'(i,j)=\max\{r, \de_{\min}(\mathfrak{p_i}, \mathfrak{p}_j)\}$ for  $i,j\in s+1$ with $i\not=j$ and $(\{i,j\}\not=\{s,t\}$.
\end{description}

Corollary \ref{cor:minmetr} provides  a set of points  $\{w_i\in \orb(\mathfrak{p}_i) : i\in s+1\}$ 
 with $w_s=v$ and $\de(w_i,w_j)=\de'(i,j)$ for all $i,j\in s+1$.  Let $\mathfrak{p}_i'$ for $ i\in s$  be the type function with $\dom(\mathfrak{p}_i')=A\cup \{w_s\}$ and $\mathfrak{p}_i\subseteq \mathfrak{p}_i'$ and $w_i\in \orb(\mathfrak{p}_i)$.

In order  to see that $\mathfrak{p}_i'$ is a \Kat function we have to check the triangles of the form $x, w_s, \mathfrak{p}_i'$ with $x\in A$, which indeed are metric because they are isometric to the triangles $x, w_s, w_i$, which are substructures of $\mathrm{M}$ and hence metric.  The \Kat functions $\mathfrak{p}_i$ are extendible into $S$ for all $1\leq i\in s$ because of the hereditary nature of the notion of being extendible into $S$.

Also  $\mathfrak{p}_i'(w_s)\geq r\geq \rank(\mathfrak{p}_i)$ for all $ i\in s$ with $i\not=t$ and $\mathfrak{p}_i\subseteq \mathfrak{p}_i'$ implying that $\rank(\mathfrak{p}_i')=\rank(\mathfrak{p}_i)$ for all $1\leq i\in s$ with $i\not=t$ and if $\rank(\mathfrak{p}_0)=r$ and $t\not= 0$ then $\rank(\mathfrak{p}_0')=\rank(\mathfrak{p}_0)=r$. Furthermore $\rank(\mathfrak{p}_t')=r^{\langle -\rangle}$ because $\mathfrak{p}_t\subseteq \mathfrak{p}_t'$ and $\dom(\mathfrak{p}_t')=\dom(\mathfrak{p}_t)\cup \{w_s\}$ and $\rank(\mathfrak{p}_t)=r$ and $\mathfrak{p}_t'(w_s)=r^{\langle -\rangle}$. Because $\mathfrak{p}_t'=\mathfrak{g}'$ we have  $\orb_{\mathrm{H}}(\mathfrak{p}_t')\subseteq S$.   Then,  according to the definitions of $\de_{\min}$ and $\de_{\max}$ and according to Corollary \ref{cor:minmetr}:
\[
\de_{\min}(\mathfrak{p}_i',\mathfrak{p}_j')=\max\{\de_{\min}(\mathfrak{p}_i, \mathfrak{p}_j),|\de(w_s, w_i)-\de(w_s,w_j)|\}=\de_{\min}(\mathfrak{p}_i,\mathfrak{p}_j) 
\]

\end{proof}

\begin{cor}\label{cor:germ2}
Let $1\leq r^{\langle -\rangle}<r\in \mathcal{B}$ and let $A$ be a finite subset of $M$ and $S\subseteq M$  and  $ s\in \omega$     and $\mathfrak{p}_i$   a  \Kat function  for every $i\in s$ so that:  
\begin{description}
\item[\textnormal{1}.]  $\dom(\mathfrak{p}_i)=A$  for all $i\in s$ and $\mathfrak{p}_i\not=\mathfrak{p}_j$ for $i\not=j$.  
\item[\textnormal{2}.]  $\rank(\mathfrak{p}_i)=r$ for all  $1\leq i\in s$.
\item[\textnormal{3}.]  $\mathfrak{p}_i$  is extendible into $S$ for all $1\leq i\in s$.
\end{description}

Then there exists an embedding $\alpha$ of\/ $\mathrm{M}$ into $\mathrm{M}$ with $\alpha(x)=x$ for all $x\in A$ and  with image $\mathrm{H}=(H; \de)$  and a finite  set  $B\subseteq H$ with $A\cap B=\emptyset$ and for every $ i\in s$ a \Kat function $\mathfrak{p}_i'$ so that:

\begin{description}
\item[$1$.] $\dom(\mathfrak{p}_i')=A\cup B$ and $\mathfrak{p}_i\subseteq \mathfrak{p}_i'$ for all $i\in s$ and $\mathfrak{p}_i'\not=\mathfrak{p}_j'$ for $i\not=j$.
\item[$2$.] $\rank(\mathfrak{p}_i')=r^{\langle -\rangle}$ for all $1\leq i\in s$.
\item[$3$.] $\orb_{\mathrm{H}}(\mathfrak{p}'_i)\subseteq S$ for all $1\leq i\in s$.
\item[$4$.] $\de_{\min}(\mathfrak{p}_i', \mathfrak{p}_j')=\de_{\min}(\mathfrak{p}_i, \mathfrak{p}_j)$
for all $i,j\in s$ with $i\not= j$.
\item[$5$.] If $\rank(\mathfrak{p}_0)=r$ and $\mathfrak{p}_0$ is extendible into $S$ then $B\subseteq H$ with $A\cap B=\emptyset$ can be chosen such that in addition to Items  1. to 4. above, also:
    \begin{description}
    \item[$1'$.]   $\rank(\mathfrak{p}_0')=r-1$. 
    \item[$2'$.] $\orb_{\mathrm{H}}(\mathfrak{p}'_0)\subseteq S$.
    \end{description}

\end{description} 
\end{cor}
\begin{proof}
Follows by induction on $s-t$ from Lemma \ref{lem:germ}.
\end{proof}

Note: Let $A$ be a finite subset of $M$  and $S\subseteq M$  and  $\mathfrak{k}$ a \Kat function with $\dom(\mathfrak{k})=A$ and let $u\in \orb(\mathfrak{k})$.  Let $\mathfrak{p}$ be a \Kat function with $\dom(\mathfrak{p})=A$ and $\rank(\mathfrak{p})=r\in \mathcal{B}$. Then,  there exists a \Kat function $\mathfrak{q}$ with $\dom(\mathfrak{q})=A\cup \{u\}$ and $\mathfrak{p}\subseteq \mathfrak{q}$ and $\mathfrak{q}(u)=l\leq r^{\langle -\rangle}$ if and only if $l\in \de(\mathfrak{k},\mathfrak{p})$ and $l\leq r^{\langle -\rangle}$ if and only if   $\de_{\min}(\mathfrak{k},\mathfrak{p})\leq l\leq r^{\langle -\rangle}$ and $l\in \mathcal{B}$. ($\de_{\max}(\mathfrak{k},\mathfrak{p})\geq r^{\langle +\rangle}$ if $r<\max\mathcal{B}$ according to Corollary~\ref{cor:range} Item 3.  and if $r=\max\mathcal{B}$ then  $\de_{\max}(\mathfrak{k},\mathfrak{p})\geq r$ according to Corollary~\ref{cor:range} Item~2.)

\begin{lem}\label{lem:germ3}
Let $A$ be a finite subset of $M$  and $S\subseteq M$  and  $\mathfrak{k}$ a \Kat function with $\dom(\mathfrak{k})=A$. Let $v\in \orb(\mathfrak{k})$ and let $\mathfrak{P}$ be a set of \Kat functions $\mathfrak{p}$ which are extendible into $S$ and  with $\dom(\mathfrak{p})=A$ and $\rank(\mathfrak{p})=r\in \mathcal{B}$ for which there exists a \Kat function $\mathfrak{q}$ with $\dom(\mathfrak{q})=A\cup \{v\}$ and $\rank(\mathfrak{q})<r$, that is a set of \Kat functions $\mathfrak{p}$ with $\de_{\min}(\mathfrak{k}, \mathfrak{p})<r$.

Then there exists an embedding $\gamma$ of\/ $\mathrm{M}$ into $\mathrm{M}$ with $\gamma(x)=x$ for all $x\in A$ and a point $u\in \orb_{\gamma(\mathrm{M})}(\mathfrak{k})$ so that $\orb_{\gamma(\mathrm{M})}(\mathfrak{q})\subseteq S$ for every \Kat function $\mathfrak{q}$ with $\dom(\mathfrak{q})=A\cup \{u\}$ and $\mathfrak{q}(u)<r$ for which there is a \Kat function $\mathfrak{p}\in \mathfrak{P}$ with $\mathfrak{p}\subseteq \mathfrak{q}$ . If\/  $\mathfrak{k}\in \mathfrak{P}$ then $u\in S$.
\end{lem}
\begin{proof}
If $\mathfrak{k}\in \mathfrak{P}$ let $s=|\mathfrak{P}|$ and $(\mathfrak{p}_i; i\in s)$ an enumeration of $\mathfrak{P}$ with $\mathfrak{k}=\mathfrak{p}_0$. If $\mathfrak{k}\not\in \mathfrak{P}$ let $s=|\mathfrak{P}|+1$ and $(\mathfrak{p}_i; i\in s)$ an enumeration of $\mathfrak{P}\cup \{k\}$ with $\mathfrak{k}=\mathfrak{p}_0$. Then $(\mathfrak{p}_i; i\in s)$ satisfies the conditions of Corollary \ref{cor:germ2}, which then supplies an embedding $\alpha$ with image $\mathrm{H}=(H, \de)$ and a set $B$ and for every $i\in s$ a \Kat function $\mathfrak{p}_i'$. Let $u\in \orb_{\mathrm{H}}(\mathfrak{p}_0)$. If $\mathfrak{k}\in \mathfrak{P}$ then $\mathfrak{k}=\mathfrak{p}_0$ is extendible into $S$ and hence $\orb(\mathfrak{k}\cap S)\not=\emptyset$. In this case let $u\in \orb(\mathfrak{k}\cap S)$.

For $i\in s$ let $f(i)$ be the set of all $l\in \mathcal{B}$ with   $\de_{\min}(\mathfrak{p}_i, \mathfrak{p}_0)\leq l\leq r^{\langle -\rangle}$.  Then,  for every $l\in f(i)$, according to Corollary \ref{cor:germ2}: $\de_{\min}(\mathfrak{p}_i', \mathfrak{p}_0')=\de_{\min}(\mathfrak{p}_i, \mathfrak{p}_0)\leq l\leq r^{\langle -\rangle}$ and hence there exists a \Kat function $\mathfrak{s}_{i,l}$ with $\dom(\mathfrak{s}_{i,l})=A\cup B\cup \{u\}$ and  $\mathfrak{p}_i'\subseteq \mathfrak{s}_{i,l}$ and $\mathfrak{s}_{i,l}(u)=l$ and $\rank(\mathfrak{s}_{i,l})=l$, because $\rank(\mathfrak{p}_i')=r^{\langle -\rangle}\geq l$. It follows from $\mathfrak{p}_i'\subseteq\mathfrak{s}_{i,l}$ and $\orb_{\mathrm{H}}(\mathfrak{p}_i')\subseteq S$ that $\orb_{\mathrm{H}}(\mathfrak{s}_{i,l})\subseteq S$.

For every $i\in s$ and $l\in f(i)$ let $\mathfrak{t}_{i,l}$ be the \Kat function with $\mathfrak{t}_{i,l}\subseteq \mathfrak{s}_{i,l}$ and $\dom(\mathfrak{t}_{i,l})=A\cup \{u\}$. It follows from $\mathfrak{p}_i\subseteq \mathfrak{t}_{i,l}$ and $\rank(\mathfrak{p}_i)=r$ that $\rank(\mathfrak{t}_{i,l})=l=\rank(\mathfrak{s}_{i,l})$. Let $i,j\in s$ and $l\in f(i)$ and $m\in f(j)$ and $a\in \mathcal{B}$ minimal with $a\geq |l-m|$ and $b\in \mathcal{B}$ maximal with $b\leq l+m$. Note that $b\leq \de_{\max}(\mathfrak{p}'_i,\mathfrak{p}'_j)\leq \de_{\max}(\mathfrak{p}_i,\mathfrak{p}_j)$. Then $\de(\mathfrak{t}_{i,l}, \mathfrak{t}_{j,m})=\de\big(\mathfrak{s}_{i,l}, \mathfrak{s}_{j,m})$, because:
\begin{align*}
&\de_{\min}(\mathfrak{t}_{i,l}, \mathfrak{t}_{j,m})=\\
&=\max\{\de_{\min}(\mathfrak{p}_i, \mathfrak{p}_j), a\}
=\max\{\de_{\min}(\mathfrak{p}_i', \mathfrak{p}_j'), a\}=\, 
\de_{\min}\big(\mathfrak{s}_{i,l}, \mathfrak{s}_{j,m})
\end{align*}
and
\begin{align*}
&\de_{\max}(\mathfrak{t}_{i,l}, \mathfrak{t}_{j,m})=\\
&=\min\{\de_{\max}(\mathfrak{p}_i, \mathfrak{p}_j),b\}=b=\min\{\de_{\max}(\mathfrak{p}_i', \mathfrak{p}_j'),b\}=\de_{\max}(\mathfrak{s}_{i,l}, \mathfrak{s}_{j,m}).
\end{align*}
Hence we can apply Theorem \ref{thm:reducegen} to obtain an embedding $\beta$ of\/ $\mathrm{H}$ to $\mathrm{H}$ with $\beta(x)=x$ for all $x\in A\cup \{u\}$ with $\orb_{\beta(\mathrm{H})}(\mathfrak{t}_{i,l})\subseteq \orb_{\mathrm{H}}(\mathfrak{s}_{i,l})\subseteq S$ for all $i\in s$ and $l\in f(i)$. Let $\gamma=\beta\circ\alpha$.
\end{proof}

\begin{thm}\label{thm:cext}
Let $\mathfrak{q}$ be a \Kat function of\/ $\mathrm{M}=(M;\de)$ with $\rank(\mathfrak{q})=r$ and $S\subseteq M$ so that $\mathfrak{q}$ is extendible into $S$.

Then there exists a copy $\mathrm{C}=(C;\de)$ of\/ $\mathrm{M}$  in $\mathrm{M}$ with $\dom(\mathfrak{q})\subseteq C$ and $\orb_{\mathrm{C}}(\mathfrak{q})\subseteq S$.
\end{thm}
\begin{proof}
Let $\mathrm{H}=(H;\de)$ be a copy of $\mathrm{M}$ in $\mathrm{M}$ with $\dom(\mathfrak{q})\subseteq H$ and $A$ a finite subset of $H$ with  $\dom(\mathfrak{q})\subseteq A$. Let $\mathfrak{k}$ be a \Kat function of $\mathrm{H}$ with $\dom(\mathfrak{k})=A$.  Let $\mathfrak{P}$ be the set of \Kat functions $\mathfrak{p}$ with $\mathfrak{q}\subseteq \mathfrak{p}$ and $\dom(\mathfrak{p})=A$ and $\rank(\mathfrak{p})=r$ and $\de_{\min}(\mathfrak{p},\mathfrak{k})<r$. (Note that if $\mathfrak{q}\subseteq \mathfrak{k}$ and $\rank(\mathfrak{k})=r$ then $\mathfrak{k}\in \mathfrak{P}$.)  Because  $\mathfrak{p}$ is extendible into $S$ for every $\mathfrak{p}\in \mathfrak{P}$,  there exists, according to Lemma \ref{lem:germ3}, an embedding $\gamma$ of $\mathrm{H}$ into $\mathrm{H}$ with $\gamma(x)=x$ for all $x\in A$  and a point $u\in \orb_{\gamma(\mathrm{H})}(\mathfrak{k})$ so that $\orb_{\gamma(\mathrm{H})}(\mathfrak{s})\subseteq S$ for every \Kat function $\mathfrak{s}$ with $\dom(\mathfrak{s})=A\cup \{u\}$ and $\mathfrak{s}(u)<r$ for which there is a \Kat function $\mathfrak{p}\in \mathfrak{P}$ with $\mathfrak{p}\subseteq \mathfrak{s}$ . If\/  $\mathfrak{k}\in \mathfrak{P}$ then $u\in S$.

It follows that the copy $\mathrm{C}$ can be constructed recursively.
\end{proof}

\section{Colouring $\mathbf{M}_{\boldsymbol{\mathcal{D}}}$}

Let $\mathcal{D}$ be a universal set of numbers with $\min(\mathcal{D}\setminus\{0\})=1$ and $\mathcal{B}$ the first block of $\mathcal{D}$, that is $\min\mathcal{B}=1$. Let $\mathrm{M}_{\mathcal{D}}=(M_{\mathcal{D}}, \de)\in \mathfrak{U}_{\mathcal{D}}$. Let $\chi: M\to \{0,1\}=2$ be a colouring  of $M_{\mathcal{D}}$ and $S_0:=\{x\in M_{\mathcal{D}} : \chi(x)=0\}$ and $S_1:=\{x\in M_{\mathcal{D}} : \chi(x)=1\}$. Then $\chi$ induces a colouring of every copy $\mathrm{M}$ of $\mathrm{M}_{\mathcal{D}}$ in $\mathrm{M}_{\mathcal{D}}$.

Let $\mathrm{M}=(M, \de)$ be a copy of $\mathrm{M}_{\mathcal{D}}$ in $\mathrm{M}_{\mathcal{D}}$.   For $E=(v_i; i\in \omega)$ an enumeration of $M$ and $\alpha$ an embedding of $\mathrm{M}$ into $\mathrm{M}$ we denote by $\alpha(E)$ the enumeration $(\alpha(v_i); i\in \omega)$ of $\alpha(M)$ and for $n\in \omega$ by $E_n$ the initial interval $(v_i; i\in n)$ of $E$.

A \Kat function $\mathfrak{t}$ is {\em monochromatic in clolour $i\in 2$} on $M$ if $\chi(x)=i$ for all $x\in \orb(\mathfrak{t})$ and $\mathfrak{t}$ is {\em monochromatic} on $M$ if there is $i\in 2$ so that $\mathfrak{t}$ is monochromatic in colour $i$ on $M$. For $r\in \mathcal{D}$, the  enumeration $E$ is {\em $r$-uniform}  from $l\in \omega$ on  if for every $n\in \omega$   with  $l<n$  every \Kat function $\mathfrak{p}$ with $\dom(\mathfrak{p})=E_n$  and $\rank(\mathfrak{p})=r$ is monochromatic.  

\vskip 4pt
\noindent
For $ r\in \mathcal{D}$ let $\mathbf{p}(r)$ be the statement:
\vskip 2pt
\noindent
$\mathbf{p}(r)$: For  every copy $\mathrm{M}=(M, \de)$ of $\mathrm{M}_{\mathcal{D}}$ in $\mathrm{M}_{\mathcal{D}}$ and every enumeration $E=(v_i; i\in \omega)$ of $M$  and    $n\in \omega$  there exists an embedding $\alpha$ of $\mathrm{M}$ into $\mathrm{M}$ with $\alpha(x)=x$ for all $x\in E_n$ and a continuation of the enumeration of $E_n$ to an enumeration  $\alpha(E)=(\alpha(v_i) : i\in \omega)$ of $\alpha(M)$ which is $r$-uniform from $n$.
\vskip 5pt

\begin{lem}\label{lem:color1}
Let $E=(v_i; i\in \omega)$ an enumeration of $M$.   Let    $n\in \omega$ and $\mathfrak{p}$ a \Kat function with $\dom(\mathfrak{p})=E_n$ and $\rank(\mathfrak{p})=1$.  Then there exists an embedding $\alpha$ of \/ $\mathrm{M}$ into $\mathrm{M}$ with $\alpha(x)=x$ for all $x\in E_n$ so that $\mathfrak{p}$ is  monochromatic. 
\end{lem}
\begin{proof}{\  }\\
\noindent
\textbf{C}ase 1: There exists a \Kat function $\mathfrak{s}$ with $\mathfrak{p}\subseteq \mathfrak{s}$ and $\chi(x)=0$ for all $x\in \orb(\mathfrak{s})$. 

Then $\rank(\mathfrak{s})=\rank(\mathfrak{p})=1$. According to Corollary~\ref{cor:reduce} with $\mathfrak{t}$ for $\mathfrak{p}$ there exists an isometry $\alpha$ with $\alpha(x)=x$ for all $x\in  \dom(\mathfrak{p})=E_n$ and $\alpha(x)\in \orb(\mathfrak{s})$ for all $x\in \orb(\mathfrak{p})$. Then:  $x\in \orb_{\alpha(\mathrm{M})}(\mathfrak{p})$ implies $x\in \orb_{\mathrm{M}}(\mathfrak{s})$ and hence $\chi(x)=0$. 
\vskip 4pt
\noindent
\textbf{C}ase 2: For every \Kat function $\mathfrak{s}$  with $\mathfrak{p}\subseteq \mathfrak{s}$ there exists a point $v\in \orb(\mathfrak{s})$ with $\chi(v)=1$. Then we construct  recursively  an embedding $\alpha$ with $\alpha(x)=x$ for all $x\in E_n$ and so that $\mathfrak{p}$ is monochromatic in colour 1 on $\alpha(E)$. 
\end{proof}

\begin{cor}\label{cor:color1}  Let $l\in \omega$.
There exists an embedding $\alpha$ of \/ $\mathrm{M}$ into $\mathrm{M}$ with $\alpha(x)=x$ for all 
$x\in E_l$  so that $\alpha(E)$ is 1-uniform from $l$. Hence $\mathbf{p}(1)$. 
\end{cor}

Let $\mathcal{D}$ be a homgen set and $\mathcal{B}$ the block of $\mathcal{D}$ with $\min(\mathcal{D}\setminus\{0\})=\min\mathcal{B}:=\mathbf{m}$ and $r\in \mathcal{B}$ with $r^{\langle -\rangle}<r$.
Let $\mathrm{M}=(M, \de)$ be a copy of $\mathrm{M}_{\mathcal{D}}$.

\begin{lem}\label{lem:eithor}
Let $1< r\in \mathcal{B}$ and $\mathbf{p}(r^{\langle -\rangle})$.   Let $\mathfrak{p}$ be a \Kat function of\/ $\mathrm{M}$ with $\rank(\mathfrak{p})=r$ which is not extendible into $S_0$.    

Then there exists a copy $\mathrm{C}=(C; \de)$ of\/ $\mathrm{M}$ in $\mathrm{M}$ with $\dom(\mathfrak{p})\subseteq C$ and $\orb_{\mathrm{C}}(\mathfrak{p})\subseteq S_1$. 
\end{lem}
\begin{proof}
There exists a copy $\mathrm{H}=(H; \de)$ of $\mathrm{M}$ in $\mathrm{M}$ with $\dom(\mathfrak{p})\subseteq H$ and a \Kat function $\mathfrak{p}'$ with $\dom(\mathfrak{p}')\subseteq H$ and $\mathfrak{p}\subseteq \mathfrak{p}'$ and $\rank(\mathfrak{p})=r$ so that $\orb_{\alpha(\mathrm{H})}(\mathfrak{g})\not\subseteq S_0$  for every embedding $\alpha$ of $\mathrm{H}$ into $\mathrm{H}$ with $\alpha(x)=x$ for all $x\in \dom(\mathfrak{p}')$ and all \Kat functions $\mathfrak{g}$ with $\dom(\mathfrak{g})\subseteq \alpha(H)$ and $\mathfrak{p}'\subseteq \mathfrak{g}$ and $\rank(\mathfrak{g})=r^{\langle -\rangle}$. We will show that $\mathfrak{p}'$ is extendible into $S_1$ on $\mathrm{H}$.

Let  $\mathrm{L}=(L; \de)$ be a copy of $\mathrm{H}$ in $\mathrm{H}$ with $\dom(\mathfrak{p}')\subseteq L$ and $\gamma$ an  embedding with $\gamma(H)=L$ and $\gamma(x)=x$ for all $x\in \dom(\mathfrak{p}')$. Let $\mathfrak{p}''$ be a \Kat function with $\dom(\mathfrak{p}'')\subseteq L$ and $\mathfrak{p}'\subseteq \mathfrak{p}''$ and $\rank(\mathfrak{p}'')=r$.  Let $v\in \orb_\mathrm{L}(\mathfrak{p}'')$ and $E=(v_i; i\in \omega)$  an enumeration of $L$  so that $\dom(\mathfrak{p}'')=E_n$ and $v=v_n$ for $n=|\dom(\mathfrak{p})''|$.  Let $\mathfrak{g}$ be the \Kat function with $\dom(\mathfrak{g})=E_{n+1}$ and $\mathfrak{p}''\subseteq \mathfrak{g}$ and $\mathfrak{g}(v_n)=r^{\langle -\rangle}$.   Because $\mathbf{p}(r^{\langle -\rangle})$ there exists an embedding $\beta$ of\/ $\mathrm{L}$ into $\mathrm{L}$ with $\beta(x)=x$ for all $x\in E_n$ so that $\beta(E)$   is $r^{\langle -\rangle}$-uniform from $n$ on. Then $\orb_{\alpha(\mathrm{H})}(\mathfrak{g})\subseteq S_0$ or $\orb_{\alpha(\mathrm{H})}(\mathfrak{g})\subseteq S_1$ for  $\alpha=\beta\circ\gamma$. Also,   $\alpha(x)=x$ for all $x\in \mathfrak{p}'$ and $\dom(\mathfrak{g})\subseteq \alpha(H)$ and $\mathfrak{p}'\subseteq \mathfrak{g}$ and $\rank(\mathfrak{g})=r^{\langle-\rangle}$. Hence $\orb_{\alpha(\mathrm{H})}(\mathfrak{g})\subseteq S_1$. 

It follows that $\mathfrak{p}'$ is extendible into $S_1$ on $\mathrm{H}$ and therefore from Theorem~\ref{thm:cext} that there is a copy $\mathrm{N}=(N;\de)$ of\/  $\mathrm{H}$ in $\mathrm{H}$ with $\dom(\mathfrak{p}')\subseteq N$ and $\orb_{\mathrm{N}}(\mathfrak{p}')\subseteq S_1$. According to Corollary \ref{cor:reduce} with $\mathfrak{p}$ for $\mathfrak{t}$ and $\mathfrak{p}'$ for $\mathfrak{s}$ and $\mathrm{N}$ for $\mathrm{M}$,  there exists an embedding $\delta$ of\/ $\mathrm{N}$ into $\mathrm{N}$ with $\delta(x)=x$ for all $x\in \dom(\mathfrak{p})$  and $\delta(x)\in \orb_\mathrm{N}(\mathfrak{p}')\subseteq S_1$. Let $\mathrm{C}=\delta(\mathrm{N})$. 
\end{proof}

\begin{lem}\label{lem:eithorcor}
Let  $\mathfrak{p}$ be a \Kat function of\/  $\mathrm{M}$ with $\rank(\mathfrak{p})=r\in \mathcal{B}$  and $2\cdot r^{\langle -\rangle}< r$ and with   $\mathbf{p}(r^{\langle -\rangle})$.  

Then there exists a copy $\mathrm{C}=(C; \de)$ of\/ $\mathrm{M}$ in $\mathrm{M}$ with $\dom(\mathfrak{p})\subseteq C$ and $\orb_{\mathrm{C}}(\mathfrak{p})\subseteq S_0$ or $\orb_{\mathrm{C}}(\mathfrak{p})\subseteq S_1$. 
\end{lem}
\begin{proof}
If $\mathfrak{p}$ is not extendible into $S_0$ the Lemma follows from Lemma~\ref{lem:eithor}. If $\mathfrak{p}$ is extendible into $S_0$ the Lemma follows from Theorem~\ref{thm:cext}
\end{proof}

\begin{cor}\label{cor:eitherorn}
Let  $r\in \mathcal{B}$  and $2\cdot r^{\langle -\rangle}< r$. (That is $r$ is not the minimum of $\mathcal{B}$. ) Then  
$\mathbf{p}(r^{\langle -\rangle})$ implies $\mathbf{p}(r)$.
\end{cor}

\begin{cor}\label{cor:fin3}
Let $\mathcal{D}$ be a universal set of numbers  and $\mathcal{B}$ the block of $\mathcal{D}$ with $\min(\mathcal{D}\setminus\{0\})=\min\mathcal{B}$ and let $\mathrm{M}=(M, \de)\in \mathfrak{U}_{\mathcal{D}}$ and $E=(v_i; i\in \omega)$ an enumeration of $M$.  Then for every coloring $\chi: M\to 2$ there exists an embedding $\alpha$ of\/ $\mathrm{M}$ into $\mathrm{M}$ so that for every $n\in \omega$ and every  \Kat function $\mathfrak{p}$ with $\dom(\mathfrak{p})=(\alpha(v_i); i\in n)$ and $\rank(\mathfrak{p})=\max\mathcal{B}$ there exists $i\in 2$ with $\chi(x)=i$ for all $x\in \orb(\mathfrak{p})$.
\end{cor}

Note that for  $\mathcal{D}$ universal  the class $\mathfrak{U}_{\mathcal{D}}$ is indivisible if an, and hence  all, $\mathrm{M}\in \mathfrak{U}_{\mathcal{D}}$ are indivisible.
\begin{thm}\label{thm:fin2}
Let $\mathcal{D}$ universal  consist of a single block. Then $\mathfrak{U}_{\mathcal{D}}$ is indivisible.
\end{thm}
\begin{proof}
Let $\mathrm{M}=(M,\de)\in \mathfrak{U}_{\mathcal{D}}$. It follows from Corollary \ref{cor:fin3} that there is a monochromatic  \Kat function $\mathfrak{p}$ of $\mathrm{M}$ with $\rank(\mathfrak{p})=\max\mathcal{D}$ and from Theorem \ref{thm:orbits} that $\restrict{\mathrm{M}}{\orb(\mathfrak{p}})$ is a copy of $\mathrm{M}$ because $\mathcal{D}_\mathfrak{p}=\{n\in \mathcal{D} : n\leq 2\cdot\max\mathcal{D}\}=\mathcal{D}$.
\end{proof}

\section{More than one euquivalence class}

Let $\mathcal{D}$ be universal  and  $\mathcal{B}$ be the block of $\mathcal{D}$ with $\min(\mathcal{D}\setminus\{0\})=\min\mathcal{B}$ and let  $\max\mathcal{B}=\mathbf{m}$ and $\mathcal{D}$ have at least two blocks. Let $\mathrm{M}=(M,\de)\in \mathfrak{U}_{\mathcal{D}}$ and  let $\sim$ denote the equivalence relation $\stackrel{\mathbf{m}}{\sim}$. For $x\in M$ denote the $\sim$-equivalence class containing $x$ by $[x]$.



\begin{lem}\label{lem:imp}
Let $A\in  M{/\negthickspace\sim}$ and $a\in A$ and $X=\{x\in M : \de(a,x)<\mathbf{m}\}$ and $C=M\setminus X$. Then $\restrict{\mathrm{M}}{C}$ is a copy of\/ $\mathrm{M}$ in $\mathrm{M}$.
\end{lem}
\begin{proof}
Let $\mathfrak{p}$ be a \Kat function of $\mathrm{M}$ with $\dom(\mathfrak{p})\subseteq C$.  According to Corollary \ref{cor:Fra1_2} we have to show that $\orb(\mathfrak{p})\cap C\not=\emptyset$.   If $\rank(\mathfrak{p})>\mathbf{m}$ Lemma \ref{lem:twocl} implies that there is a point $b\in \orb(\mathfrak{p})$ with $b\not\in A$  and hence $\orb(\mathfrak{p})\cap C\not=\emptyset$. Let $\rank(\mathfrak{p})\leq \mathbf{m}$.  Let $c\in \dom(\mathfrak{p})$ with $\mathfrak{p}(c)\leq \mathbf{m}$. Then $\orb(\mathfrak{p})\subseteq [c]$. Hence if $c\not\in A$ then $\orb(\mathfrak{p})\subseteq C$. Let $c\in A$. It follows that $\orb(\mathfrak{p})\subseteq A$ and that $\mathfrak{p}(x)\leq \mathbf{m}$ for all $x\in \dom(\mathfrak{p})\cap A$ and that $\mathfrak{p}(x)\geq 2\cdot\mathbf{m}$ for all $x\in \dom(\mathfrak{p})\setminus A$.

Let $b\in \orb(\mathfrak{p})$ and  $\mathfrak{t}$ the \Kat function with $\dom(\mathfrak{t})=\mathfrak{p}\cup \{a\}$ and $\mathfrak{p}\subseteq \mathfrak{t}$ and $b\in \orb(\mathfrak{t})$. Then $\mathfrak{t}(a)\leq \mathbf{m}$. Let $\mathfrak{q}$ be the type function with $\dom(\mathfrak{q})=\dom(\mathfrak{p})\cup \{a\}$ and $\mathfrak{p}\subseteq \mathfrak{q}$ and $\mathfrak{q}(a)=\mathbf{m}$. In order to check that $\mathfrak{q}$ is a \Kat function we use inequality (1) from Section 2.  Because $\mathfrak{p}$ is a \Kat function it remains to check that $|\mathfrak{q}(a)-\mathfrak{q}(x)|\leq \de(a,x)\leq \mathfrak{q}(a)+\mathfrak{q}(x)$ for all $x\in \dom(\mathfrak{q})$. That is that:
\begin{align}
|\mathbf{m}-\mathfrak{p}(x)|\leq \de(a,x)\leq \mathbf{m}+\mathfrak{p}(x).
\end{align}
If $x\in A$ then $x\in C\cap A$ and  $\de(a,x)=\mathbf{m}$ and then Inequaltiy~(4) holds. If $x\not\in A$ it follows from the fact that $\mathfrak{t}$ is a \Kat function and that $\mathfrak{t}(x)=\mathfrak{p}(x)$ and and hence from Inequality (1) that:
\begin{align}
|\mathfrak{t}(a)-\mathfrak{p}(x)|\leq \de(a,x)\leq \mathfrak{t}(a)+\mathfrak{p}(x).
\end{align}
Because $\mathfrak{p}(x)>\mathbf{m}\geq\mathfrak{t}(a)$ Inequality (3) implies Inequality (2). 
\end{proof}

Let $E=(v_i; i\in \omega)$ be an enumeration of $M$   and $E_n=(v_i; i\in n)$ for $n\in \omega$. The element $v_n\in M$ is {\em initial} if $[v_n]\cap E_n=\emptyset$.  For $v_n$ an initial point let $\mathfrak{e}'_n$ be the \Kat function with $\dom(\mathfrak{e}'_n)=E_n$ and $v_n\in \orb(\mathfrak{e}'_n)$. Let $\mathfrak{e}_n$ be the \Kat function with $\dom(\mathfrak{e}_n)=E_{n+1}$ and $\mathfrak{e}'_n\subseteq \mathfrak{e}_n$ and $\mathfrak{e}_n(v)=\mathbf{m}$. (Because $v_n$ is initial the rank of $\mathfrak{e}'$ is larger than $\mathbf{m}$ and hence it follows from Corollary \ref{cor:orbits} that $\mathfrak{e}_n$ is indeed a \Kat function.) Then $\rank(\mathfrak{e}_n)=\mathbf{m}$.

Let $\chi: M\to 2$ be a colouring of $M$. Then we obtain from Corollary~\ref{cor:fin3} that there exists an embedding $\alpha$ of $\mathrm{M}$ into $\mathrm{M}$ so that for every initial point $v_n$ the \Kat function $\alpha(\mathfrak{e}_n)$ is monochromatic. That is in the equivalence class $[\alpha(v_n)]$ the set of points of distance $\mathbf{m}$ from $\alpha(v_n)$ is monochromatic. It follows then from Lemma \ref{lem:imp} that by removing the points in $\alpha(M)$ at distance to $v_n$ less than $\mathbf{m}$ we obtain a copy of $\mathrm{H}=(H, \de)$ of $\alpha(\mathrm{M})$ in $\alpha(\mathrm{M})$ in which  the $\sim$-equivalence class $[v_n]\cap H$ of $H$ is monochromatic. Hence we obtained:
\begin{lem}\label{lem:monoequ}
Let $\mathcal{D}$ be universal  and  $\mathcal{B}$ be the block of $\mathcal{D}$ with $\min(\mathcal{D}\setminus\{0\})=\min\mathcal{B}$ and let  $\max\mathcal{B}=\mathbf{m}$ and $\mathcal{D}$ have at least two blocks. Let $\mathrm{M}=(M,\de)\in \mathfrak{U}_{\mathcal{D}}$ and  let $\sim$ denote the equivalence relation $\stackrel{\mathbf{m}}{\sim}$. Let $\chi: M\to 2$ be a colouring of $M$. Then there exists a copy of $\mathrm{M}$ in $\mathrm{M}$ in which every $\sim$-equivalence class is monochromatic. 
\end{lem}

For $\mathrm{M}=(M, \de)\in \mathfrak{U}_{\mathcal{D}}$ let $\mathrm{M}/\negthickspace\sim$ be the metric space with $M/\negthickspace\sim$ as set of points and distance function $\de_{\min}$ with $\de_{\min}(A,B)=\min\{\de(x,y) : x\in A, y\in B\}$, see Lemma \ref{lem:trdist}. Let $\mathcal{D}_{\min}$ be the set of distances in  $\mathrm{M}/\negthickspace\sim$.

\begin{lem}\label{lem:demin}
If $\mathcal{D}$ is a universal set of numbers then $\mathcal{D}_{\min}$ is universal and for $\mathrm{M}=(M,\de)\in \mathfrak{U}_{\mathcal{D}}$ the metric space $\mathrm{M}/\negthickspace\sim=(M/\negthickspace\sim, \de_{\min})\in \mathfrak{U}_{\mathcal{D}_{\min}}$.  If $\mathrm{N}=(N;\de_{\min})$ is a copy of $\mathrm{M}/\negthickspace\sim$ in $\mathrm{M}/\negthickspace\sim$ then $\bigcup N$ induces a copy of $\mathrm{M}$ in $\mathrm{M}$. 
\end{lem}
\begin{proof}
Let $\mathfrak{p}$ be a \Kat  function of $\mathrm{M}/\negthickspace\sim$ with $\mathfrak{p}(X)\in \mathcal{D}_{\min}$ for all $X\in \dom(\mathfrak{p})$.  According to Lemma \ref{lem:embedmin} there exists an isometry $\beta:\dom(\mathfrak{p})\to M$ with $\beta(X)\in X$ for all $X\in \dom(\mathfrak{p})$.  Let $\mathfrak{p}'$ be the \Kat function of $\mathrm{M}$ with $\dom(\mathfrak{p}')=\{\beta(X) : X\in \dom(\mathfrak{p}')\}$ and $\mathfrak{p}'(\beta(X))=\mathfrak{p}(X)$. Then $\mathfrak{p}'$ is a \Kat function of $\mathrm{M}$. Let $a\in \orb(\mathfrak{p}')$ then $[a]\in \orb(\mathfrak{p})$.  Hence $\mathrm{M}/\negthickspace\sim\in \mathfrak{U}_{\mathcal{D}_{\min}}$ according to Theorem \ref{thm:Fra1}.

Let $\mathfrak{p}$ be a \Kat function of $\mathrm{M}$ with $\dom(\mathfrak{p})\subseteq \bigcup N$. Then $\mathfrak{p}$ has a realization $a\in M$. There exists an isometry $\alpha$ of  $\bigl(\{[x]_r \mid x\in \dom(\mathfrak{p})\}\cup \{[a]_r\}\bigr)$ into $N$ which is the identity on $\{[x]_r \mid x\in \dom(\mathfrak{p})\}$ and maps $[a]_r$ into $N$. Let $b\in \alpha([a]_r)$. Let $\mathfrak{q}$ be the type  function with $\mathfrak{p}\subseteq \mathfrak{q}$ and $\dom(\mathfrak{q})=\dom(\mathfrak{p}\cup \{b\}$ and $\mathfrak{q}(b)=r$. Then $\mathfrak{q}$ is a \Kat function and has a realization in $[a]_r$.

\end{proof}

\begin{thm}\label{thm:final}
Let $\mathcal{D}$ be a universal set of numbers  and   $\mathrm{M}=(M,\de)\in \mathfrak{U}_{\mathcal{D}}$.  Then $\mathrm{M}$ is indivisible.  
\end{thm}
\begin{proof}
Using Lemma \ref{lem:monoequ} and  Lemma \ref{lem:demin} the Theorem follows by induction on the number of blocks in $\mathcal{D}$ with Theorem \ref{thm:fin2} covering the case of a single block.
\end{proof}

 \end{document}